%% file: compiling_file.tex
\documentclass[12pt]{amsart}

\usepackage{amsmath,amsfonts,amsthm,amssymb}
\usepackage{hyperref}
\usepackage{tikz}
\usetikzlibrary{arrows.meta, positioning, calc}

\newcommand{\nodedistance}{0.5}
\newcommand{\circleinnersep}{1.6}
\renewcommand{\arraystretch}{1.2}

\makeatletter
\newsavebox{\@brx}
\newcommand{\llangle}[1][]{\savebox{\@brx}{\(\m@th{#1\langle}\)}%
	\mathopen{\copy\@brx\kern-0.5\wd\@brx\usebox{\@brx}}}
\newcommand{\rrangle}[1][]{\savebox{\@brx}{\(\m@th{#1\rangle}\)}%
	\mathclose{\copy\@brx\kern-0.5\wd\@brx\usebox{\@brx}}}
\makeatother

\usepackage{geometry}
\usepackage{pdflscape}
\usepackage[width=.75\textwidth]{caption}

\usepackage{a4wide}
\usepackage{enumerate}
\usepackage{multicol, multirow, longtable}
\newtheorem{theorem}{Theorem}[section]
\newtheorem{lemma}[theorem]{Lemma}
\newtheorem{proposition}[theorem]{Proposition}
\newtheorem{corollary}[theorem]{Corollary}

\newtheorem{definition}[theorem]{Definition}

\usepackage{parskip}
\newtheorem{theo}{Theorem}

\theoremstyle{remark}
\newtheorem{remark}[theorem]{Remark}

\newtheorem*{claim*}{Claim}

\numberwithin{equation}{section}

\newcommand{\C}{\ensuremath{\mathbb{C}}}
\newcommand{\R}{\ensuremath{\mathbb{R}}}
\newcommand{\HH}{\ensuremath{\mathbb{H}}}
\newcommand{\g}[1]{\ensuremath{\mathfrak{#1}}}

\DeclareMathOperator{\codim}{codim}

\DeclareMathOperator{\Ad}{Ad}
\DeclareMathOperator{\ad}{ad}

\DeclareMathOperator{\spann}{span}
\DeclareMathOperator{\diag}{diag}
\DeclareMathOperator{\Ric}{Ric}

\DeclareMathOperator{\rank}{rank}

\renewcommand{\H}{\ensuremath{\mathbb{H}}}

\newcommand{\spin}[1]{\ensuremath{\mathsf{Spin}_{#1}}}
\newcommand{\su}[1]{\ensuremath{\mathsf{SU}_{#1}}}
\renewcommand{\u}[1]{\ensuremath{\mathsf{U}_{#1}}}
\newcommand{\so}[1]{\ensuremath{\mathsf{SO}_{#1}}}
\newcommand{\oo}[1]{\ensuremath{\mathsf{O}_{#1}}}
\renewcommand{\sp}[1]{\ensuremath{\mathsf{Sp}_{#1}}}
\newcommand{\gr}[3]{\ensuremath{\mathsf{G}_{#1}(#3^{#2})}}
\newcommand{\grp}[3]{\ensuremath{\mathsf{G}^+_{#1}(#3^{#2})}}

\makeatletter
\g@addto@macro\bfseries{\boldmath}
\makeatother

	\DeclareMathOperator{\trace}{trace}
	
	\mathcode`l="8000
	\begingroup
	\makeatletter
	\lccode`\~=`\l
	\DeclareMathSymbol{\lsb@l}{\mathalpha}{letters}{`l}
	\lowercase{\gdef~{\ifnum\the\mathgroup=\m@ne \ell \else \lsb@l \fi}}%
	\endgroup

	\newcommand*{\defeq}{\mathrel{\vcenter{\baselineskip0.5ex \lineskiplimit0pt
				\hbox{\scriptsize.}\hbox{\scriptsize.}}}%
		=}
	
\newcommand{\NN}{\mathbb{N}}
\newcommand{\RR}{\mathbb{R}}

\newcommand{\CC}{\mathbb{C}}

\newcommand{\sph}{\mathbb{S}}
\newcommand{\RP}{\mathbb{RP}}
\newcommand{\CP}{\mathbb{CP}}
\newcommand{\HP}{\mathbb{HP}}

\newcommand{\ZZ}{\mathbb{Z}}

\newcommand{\sss}{\ensuremath{\mathsf{S}}}
	\DeclareMathOperator{\scal}{scal}
	
		\DeclareMathOperator{\Mod}{mod}
\begin{document}
\title[Infinite families of manifolds of positive intermediate Ricci curvature]{Infinite families of manifolds of\\positive $k^{\rm th}$-intermediate Ricci curvature with $k$ small}
%    authors information
\author[M.~Dom\'{\i}nguez-V\'{a}zquez]{Miguel Dom\'{\i}nguez-V\'{a}zquez}
\address{Departamento de Matem\'aticas, Universidade de Santiago de Compostela, Spain.}
\email{miguel.dominguez@usc.es}

\author[D.~Gonz\'alez-\'Alvaro]{David Gonz\'alez-\'Alvaro}
\address{ETSI de Caminos, Canales y Puertos, Universidad Polit\'ecnica de Madrid, Spain.}
\email{david.gonzalez.alvaro@upm.es}

\author[L.~Mouill\'e]{Lawrence Mouill\'e}
\address{Department of Mathematics, Rice University, USA.}
\email{lm59@rice.edu}
%\date{\today}
%\ [Preliminary version, please do not distribute.]

\begin{abstract}
Positive $k^{\rm th}$-intermediate Ricci curvature on a Riemannian $n$-manifold, to be denoted by $\Ric_k>0$, is a condition that interpolates between positive sectional and positive Ricci curvature (when $k =1$ and $k=n-1$ respectively). In this work, %we consider various metrics on compact homogeneous spaces and study the minimum $k$ for which they satisfy $\Ric_k>0$. 
we produce many examples of manifolds of $\Ric_k>0$ with $k$ small by examining symmetric and normal homogeneous spaces, along with certain metric deformations of fat homogeneous bundles. As a consequence, we show that every dimension $n\geq 7$ congruent to $3\,\Mod 4$ supports infinitely many closed simply connected manifolds of pairwise distinct homotopy type, all of which admit homogeneous metrics of $\Ric_k>0$ for some $k<n/2$. We also prove that each dimension $n\geq 4$ congruent to $0$ or $1\,\Mod 4$ supports closed manifolds which carry metrics of $\Ric_k>0$ with $k\leq n/2$, but do not admit metrics of positive sectional curvature. 
\end{abstract}

\thanks{The first author has been supported by projects PID2019-105138GB-C21 (AEI/FEDER, Spain), ED431C 2019/10, ED431F 2020/04 (Xunta de Galicia, Spain) and by the	Ram\'{o}n y Cajal program of the Spanish State Research Agency. The second author received support from MINECO grant MTM2017-85934-C3-2-P. The third author was supported in part by NSF Grant DMS-1612049.}


\maketitle

%\tableofcontents

%%%%%%%%%%%%%%%%%%%%%%%%%%%%%%%%%%%%%%%%%%%%%%%%%%%%%%%%%%%%%%%%%%%%%%%%%%%%

\input{introduction_split.tex}

\input{preliminaries.tex}

\input{homogeneous.tex}

\input{fat_bundles.tex}

\input{tables.tex}

%%%%%%%%%%%%%%%%%%%%%%%%%%%%%%%%%%%%%%%%%%%%%%%%%%%%%%%%%%%%%%%%

\end{document}

%% file: introduction_split.tex
A complete Riemannian manifold $(M,g)$ is said to have \emph{positive $k^{th}$-intermediate Ricci curvature}, to be denoted by $\Ric_k>0$, if for every point $p\in M$ and every set of $k+1$ orthonormal tangent vectors $x,e_1,\dots,e_k\in T_p M$, the sum of the sectional curvatures of the planes spanned by $x,e_i$ is positive, i.e.
\[
\sum_{i=1}^k \sec_g(x,e_i)>0.
\]
If $n=\dim M$, then $\Ric_1>0$ and $\Ric_{n-1}>0$ correspond to the classical conditions of positive sectional curvature $\sec>0$ and positive Ricci curvature $\Ric>0$, respectively. Note that $\Ric_k>0$ implies $\Ric_{l}>0$ for any $l\geq k$; thus, the smaller the $k$, the more restrictive the condition $\Ric_k>0$. We refer the reader to Section~\ref{S:preliminaries} for background, context, and a more general definition of $\Ric_k$. 

At the present state of knowledge, no purely topological obstructions are known for a closed manifold $M$ to admit a metric of $\Ric_k>0$ with $k>1$, except those already obstructing the existence of metrics of $\Ric>0$ or positive scalar curvature (e.g.\ Bonnet-Myers Theorem or index-theoretical results by Atiyah-Singer-Lichnerowicz-Hitchin; see \cite{Zi07}).

As the dimension of the manifold under consideration increases, the condition $\sec>0$ appears to be stronger and stronger. In fact, we only know of two dimensions supporting infinitely many simply connected closed manifolds of $\sec>0$ of pairwise distinct homotopy type ($7$ and $13$, given by Eschenburg and Bazaikin spaces respectively). Furthermore, in dimensions above $24$, the only known examples are compact rank one symmetric spaces. In contrast, the condition $\Ric>0$ appears to be weaker and weaker as the dimension grows. In particular, every dimension $\geq 4$ is known to support infinitely many simply connected closed manifolds of $\Ric>0$ of pairwise distinct real homotopy type \cite{SY91}. 

%To the best of our knowledge, no such infiniteness results are known for intermediate conditions. 
%In \marginpar{{\color{blue} \tiny I have a suggested replacement of this and the next few paragraphs below.}} this article, we consider various metrics on symmetric and homogeneous spaces, and we study the minimum $k$ for which $\Ric_k>0$. 

The main goal of this article is to generate examples of manifolds of $\Ric_k>0$. An elementary way to produce new examples is to take the Riemannian product of two manifolds $M^n=N_1\times N_2$. However, such a product metric will be of $\Ric_k>0$ only for $k>\frac{n}{2}$, even if both factors are of $\sec >0$; see Proposition~\ref{PROP:products}. In order to distinguish our examples from such trivial constructions, we put primary focus in this article on spaces $M^n$ of $\Ric_{k}>0$ with $k\leq \frac{n}{2}$ and, more generally, on spaces of $\Ric_{k}>0$ where $k$ is small compared to their dimension. The following theorem is a particular instance of the results we derive in this work.

\begin{theo}\label{THM:main_theorem_existence_short}
	Every dimension $n\geq 7$ congruent to $3\,\Mod 4$ supports infinitely many closed simply connected manifolds of pairwise distinct homotopy type, all of which admit homogeneous metrics of $\Ric_k>0$ for some $k<n/2$.
\end{theo}

The families of spaces referenced in Theorem~\ref{THM:main_theorem_existence_short}, which  are generalized versions of the so-called Aloff-Wallach spaces from \cite{AW75}, are defined in \eqref{eq:Wpq} below.
The reader can find the precise $k$ for each of the spaces and a discussion about their topology in Theorem~\ref{THM:homogeneous_simple_version} and the paragraph following it, respectively. Besides the spaces in Theorem~\ref{THM:main_theorem_existence_short}, we construct many other manifolds of $\Ric_k>0$ for some $k< \frac{n}{2}$, including the unit tangent bundles or projectivized tangent bundles of compact rank one symmetric spaces. 

A number of results is involved in our constructions, and they all can be found in Section~\ref{SEC:tools_results}; we believe each of them is of independent interest. Our methods generalize those that were used to construct all homogeneous spaces of $\sec >0$, whose classification was achieved in the 1970s. In particular, we will study curvature properties of symmetric spaces, normal homogeneous spaces, and certain homogeneous spaces that fiber over symmetric spaces via fat homogeneous bundles.

Because we study homogeneous metrics here, we have many isometric actions at our disposal. By taking appropriate quotients, we obtain many manifolds of $\Ric_k>0$ which, by Synge's theorem, cannot carry metrics of $\sec >0$. Subsequently, we establish the following result, which we prove in Subsection~\ref{SS:Ric2}. 

\begin{theo}\label{THM:main_theorem_quotients}
	Each dimension $n\geq 4$ congruent to $0$ or $1\,\Mod 4$ supports (non-simply connected) closed manifolds which admit metrics of $\Ric_k>0$ for some $k\leq n/2$, but not metrics of $\sec>0$.
\end{theo}

\section{Summary of Results}\label{SEC:tools_results}

In this section, we discuss the main results of this article, put them into context, and examine their implications. 
Specifically, we deal with normal homogeneous and symmetric spaces in Subsection~\ref{SS:normal_homogeneous_symmetric}, we discuss fat homogeneous bundles in Subsection~\ref{SS:fat_homogeneous_bundles}, and we examine some important consequences in relation to non-simply connected examples, manifolds of $\Ric_2>0$, and Riemannian submersions in Subsection~\ref{SS:Ric2}.

\subsection{Normal homogeneous and symmetric spaces}\label{SS:normal_homogeneous_symmetric}
We start with the class of compact homogeneous spaces $G/H$ with \textit{finite fundamental group} and $G$ connected. % (since otherwise they cannot admit metrics of $\Ric>0$ by Bonnet-Myers Theorem). %For such a space $G/H$ one can always find a semisimple Lie group acting transitively on $G/H$. Thus, unless otherwise stated we will assume throughout the whole article that $G$ is semisimple. 
It is a classical result that \emph{normal} homogeneous metrics on $G/H$ (i.e.\ descending from a bi-invariant metric on $G$) are of $\Ric>0$; see e.g.\ the work of Nash \cite[Proposition~3.4]{Na79} or Berestovskii \cite[Theorem~1]{Be95}. As we shall see in Theorem~\ref{THM:extension_Nash_Ber} below, for most spaces, normal metrics are actually ``more curved'' than simply $\Ric>0$ from the perspective of positive $k^\mathrm{th}$-intermediate Ricci curvature. 

First, we observe that for any two normal metrics $g$, $g'$ on $G/H$, the minimum $k$ for which $g$ is of $\Ric_k>0$ agrees with the minimum $k$ for which $g'$ is of $\Ric_k>0$; see Corollary~\ref{cor:independence}. Hence, in order to facilitate our discussion, we introduce the following notation:
\[
b(G/H)\defeq \min\{k\in\NN : \Ric_k>0 \text{ for a normal homogeneous metric } g \text{ on }G/H\}.
\]
Note that the integer $b(G/H)$ depends on the pair $(G,H)$ rather than on the diffeomorphism type of $G/H$; see Remarks~\ref{REM:b_depends_on_pair} and \ref{REM:b_depends_on_pair_2} below. With this notation, the result of Nash-Berestovskii can be stated with the inequality $b(G/H)\leq \dim G/H-1$. We extend this result by characterizing the equality case:

\begin{theo}\label{THM:extension_Nash_Ber}
Let $G/H$ be a normal homogeneous space for a semisimple Lie group $G$. Then $b(G/H)=\dim G/H - 1$ if and only if $G/H$ locally splits off a factor isometric to a round $2$-sphere. 
\end{theo}

%Moreover,
%	\[	
%	b(G/H)\leq \frac{b(G)+\dim G/H-\dim H}{2}.
%	\]
%(The number $b(G)$ shall be computed in Theorem~\ref{th:gaps_sym} for any compact Lie group $G$.)

%This includes in particular the computation of $b(G)=b((G\times G)/\Delta G)$ for any compact semisimple Lie group $G$, as announced in Theorem~\ref{THM:extension_Nash_Ber} (see Remark~\ref{REM:bG_equals_bGxG_DeltaG}).

Theorem~\ref{THM:extension_Nash_Ber} can be seen as a classification of those homogeneous spaces $G/H$ with $b(G/H)=\dim G/H -1$. In general, it would be a large undertaking to determine $b(-)$ for every homogeneous space. Indeed, only the opposite and most restrictive case, $b(G/H)=1$, is fully known. Berger classified all pairs $(G,H)$ for which $G/H$ admits a normal homogeneous metric of $\sec>0$ \cite{Be61}, with an omission corrected by Wilking \cite{Wi99}. The rather short list consists of compact rank one symmetric spaces and three other spaces: the so-called Berger spaces $B^7\defeq\so{5}/\so{3}$ and $B^{13}\defeq\su{5}/\sp{2}\u{1}$, and Wilking's space $(\so{3}\times \su{3})/\u{2}$. For descriptions of these spaces,  see, for example, \cite{WZ18}.

%(with possibly various presentations $G/H$ as homogeneous spaces) 

\begin{remark}\label{REM:b_depends_on_pair}
Wilking's positively curved normal homogeneous space $(\so{3}\times \su{3})/\u{2}$ is diffeomorphic to the Aloff-Wallach space $W_{1,1}^{7}\defeq\su{3}/\sss(\u{1}^{1,1}\times \u{1})$; see \eqref{EQ:circle_torus} for a description of $\u{1}^{1,1}$. However, $W_{1,1}^{7}$ does not admit a positively curved normal homogeneous metric (i.e.\ descending from a bi-invariant metric on $\su{3}$). Thus, $b(W_{1,1}^{7})>1 = b((\so{3}\times \su{3})/\u{2})$.
\end{remark}

%There is a subclass of homogeneous spaces, though, for which it is manageable to determine $b(-)$; namely that of symmetric spaces. 

There are two subclasses of homogeneous spaces, though, for which it is manageable to determine $b(-)$; namely symmetric spaces and spaces of the form $(G\times G\times G)/\Delta G$. 
We will discuss symmetric spaces here. 
For the spaces $(G\times G\times G)/\Delta G$, see the discussion around Theorem~\ref{THM:Wilkings_products} below.

In Section~\ref{SEC:symmetric_spaces} we determine $b(G/K)$ for every symmetric space $G/K$ of compact type. In particular, we list the values $b(G/K)$ for the irreducible symmetric spaces in Table~\ref{table:b_symmetric} at the end of this article. Very recently, during the writing of this paper, Amann, Quast and Zarei obtained the same result independently, and used it to study the higher connectedness of symmetric spaces~\cite{AQZ20}.
%We note that  $b(G/K)\leq 2$ if and only if $\rank(G/K)=1$, or equivalently, $b(G/K)=1$, see Theorem~\ref{THM:observations_table_symmetric_spaces}(a).

In order to determine $b(G/K)$ for each symmetric space $G/K$, the basic observation is that $b(G/K)=\max_{x\in\g{p}\setminus\{0\}} \dim Z_\g{p}(x)$, where $\g{p}$ is the orthogonal complement of $\g{k}$ in $\g{g}$ (the Lie algebras of $K$ and $G$, respectively), and $Z_\g{p}(x)$ is the centralizer of $x$ in $\g{p}$. The centralizers that may have maximal dimension determine some particularly nice class of totally geodesic submanifolds of $G/K$. By standard theory of root systems, one can calculate the dimensions of such totally geodesic submanifolds, thus deriving the value of $b(G/K)$. The proof in \cite{AQZ20} is, in essence, equivalent to ours, the main difference being that we make use of the Dynkin diagrams instead of the explicit description of the roots. 

We now summarize the main facts of our study of symmetric spaces. Throughout this article, $\grp{p}{p+q}{\R}\defeq\so{p+q}/(\so{p}\times\so{q})$ denotes the oriented real Grassmannians, and $\gr{p}{p+q}{\C}\defeq\su{p+q}/\mathsf{S}(\u{p}\times\u{q})$, $\gr{p}{p+q}{\H}\defeq\sp{p+q}/(\sp{p}\times\sp{q})$ the complex and quaternionic Grassmannians. 

%extracted from Section~\ref{SEC:symmetric_spaces}.

% (see the end of Section~\ref{SEC:symmetric_spaces}):

%[Moreover, if $G/K$ is irreducible then all homogeneous metrics are normal, so the computation covers all homogeneous metrics.]

\begin{theo}\label{THM:observations_table_symmetric_spaces}
Let $M=G/K$ be a symmetric space of compact type. If $M$ is irreducible, then $b(M)$ is given in Table~\ref{table:b_symmetric} at the end of this article. If $M$ is reducible, and $\widetilde{M}=M_1\times \dots \times M_s$ is the decomposition into irreducible factors of the universal cover of $M$, then $b(M)=\max\{b(M_j)+\dim M- \dim M_j:j=1,\dots,s\}$. %The values $b(M)$ for reducible $M$ can be obtained by combining Table~\ref{table:b_symmetric} and Proposition~\ref{prop:products_homogeneous}. 
Moreover:
\begin{enumerate}[\rm(a)]
	\item $b(M)\geq 2\rank M-1$, with equality if and only if $M$ has rank one or its universal cover is $\su{3}/\so{3}$, $\grp{2}{5}{\R}$, $\mathsf{G}_2/\so{4}$ or a finite product of $2$-spheres. In particular, $b(M)$ is never equal to $2$. \label{item:Da}
	
	\item $b(M)\leq \dim M-3$ if and only if $M$ does not locally split off a $2$-sphere, a $3$-sphere, or a Wu space $\su{3}/\so{3}$.\label{item:Db}
	
\item \label{item:Dc} $b(M)\leq (\dim M)/2$ if and only if $M$ is irreducible and either
		\begin{enumerate}[\rm(i)]
			\item $M$ has rank $1$ or $2$ but $M\not\cong\su{3}/\so{3}$, or
			\item $M$ is isomorphic to one of the following spaces of rank $3$ or $4$: $\gr{3}{6}{\C}$, $\gr{3}{6}{\H}$, $\mathsf{E}_7/\mathsf{E}_6\u{1}$, $\so{12}/\u{6}$, $\so{14}/\u{7}$, and any space with Dynkin diagram of $\mathsf{F}_4$ type (i.e.\ $\mathsf{F}_4/\sp{3}\sp{1}$, $\mathsf{E}_6/\su{6}\su{2}$, $\mathsf{E}_7/\spin{12}\sp{1}$, $\mathsf{E}_8/\mathsf{E}_7\sp{1}$, $\mathsf{F}_4$).
		\end{enumerate}
			
\item $b(M)\leq (\dim M)/3$ if and only if either $M$ is of rank $1$ but  $M\not\cong\mathbb{S}^2$, or $M\cong\mathsf{G}_2$.\label{item:Dd}

\item\label{item:De} All irreducible $M$ with $3\leq b(M)\leq 6$ are listed in Table~\ref{table:low_b_symmetric} in Section~\ref{SEC:symmetric_spaces}.

\item\label{item:Df} If $M$ has rank $2$, the value of $b(M)$ is given in Table~\ref{table:b_symmetric_rank2} in Section~\ref{SEC:symmetric_spaces}.
\end{enumerate}
\end{theo}

%In Section~\ref{SEC:symmetric_spaces}, for the convenience of the reader, we provide a list of all irreducible compact symmetric spaces $G/K$ with $3\leq b(G/K)\leq 6$ in Table~\ref{table:low_b_symmetric}, and we present a list of the values $b(G/K)$ for all irreducible compact symmetric spaces of rank $2$ in Table~\ref{table:b_symmetric_rank2}.
%Throughout this article, $\grp{p}{p+q}{\R}=\so{p+q}/\so{p}\times\so{q}$ denotes the oriented real Grassmannians, and $\gr{p}{p+q}{\C}=\su{p+q}/\mathsf{S}(\u{p}\times\u{q})$, $\gr{p}{p+q}{\H}=\sp{p+q}/\sp{p}\times\sp{q}$ the complex and quaternionic Grassmannians. 

%As a service to the reader we provide Tables~\ref{table:low_b_symmetric} and \ref{table:b_symmetric_rank2} in Section~\ref{SEC:symmetric_spaces} with those irreducible symmetric spaces with $3\leq b(G/K)\leq 6$ and with the values for those $G/K$ of rank $2$, respectively. 

We remark that if we consider general homogeneous (non-symmetric) metrics $g$ and replace $b(M)$ with the minimal $k$ for which $g$ satisfies $\Ric_k>0$, then Items~\eqref{item:Da} to \eqref{item:Dd} in Theorem~\ref{THM:observations_table_symmetric_spaces} would not hold. For example, the rank $2$ reducible symmetric space $\sph^3\times\sph^3$ admits a homogeneous metric of $\Ric_2>0$; see Remark~\ref{REM:b_depends_on_pair_2}. We also note that Item \eqref{item:Db} does not hold for normal metrics on homogeneous spaces. Indeed, in Subsection~\ref{SEC:Witte_spaces}, we will observe that every odd dimension $2n+3\geq 7$ contains infinitely many homogeneous spaces $G/H$ of distinct homotopy type with $b(G/H)=\dim G/H -2$ that do not locally split off any factor. Moreover, building upon the results in \cite{DKT18}, we conclude that for many of them, the moduli space  of metrics of $\Ric_{\dim G/H -2}>0$ has infinitely many connected components.

%$\mathcal{M}_{\Ric_{\dim G/H -2}>0}(G/H)$

%The class of homogeneous spaces is contained in that of biquotients, and more generally of base spaces of Riemannian submersions $G\to (M,q)$ where the Lie group $G$ is endowed with a bi-invariant metric. In this context, the numbers $b(G)$ are specially interesting, since the base space $M$ satisfies $\Ric_{b(G)}(M,q)>0$ provided $\dim M>b(G)$ \cite[Corollary~1.5]{Mo19b}. For an arbitrary base space $(M,q)$ of finite fundamental group we review a proof of Wilking showing that $\Ric(M,q)>0$, see Section~\ref{SEC:normal_biquotients}.

We now discuss normal metrics on homogeneous spaces of the form $(G\times G\times G)/\Delta G$, where $G$ is a compact semisimple Lie group. The motivation for considering these spaces comes from a construction of Burkhard Wilking, shared with the third author in a personal communication, that the Cheeger deformation of the product of round metrics on $\sph^3\times\sph^3$ via the diagonal action by $\Delta\sph^3 <\sph^3\times\sph^3$ is of $\Ric_2>0$. We refer the reader unfamiliar with Cheeger deformations to \cite[Section~2]{Zi07}.
Cheeger deformations do not generally result in homogeneous metrics, however in this particular case Wilking's metric on $\sph^3\times\sph^3$ is isometric to a normal homogeneous metric on $(\sph^3\times\sph^3\times\sph^3)/\Delta \sph^3$. His construction generalizes as follows.

\begin{theo}\label{THM:Wilkings_products}
For any compact semisimple Lie group $G$, it holds that:
\[
b\left(\frac{G\times G\times G}{\Delta G}\right)=2b(G).
\]
Consequently, $G\times G$ admits a metric of $\Ric_{k}>0$ for $k=2b(G)$ which is left-invariant and invariant under right $\Delta G$-diagonal multiplication.
\end{theo}

We will explore some consequences of the resulting metric on $G\times G$ in Subsection~\ref{SS:Ric2}. Observe that the spaces $M\defeq G^3/\Delta G$ for which $b(M)\leq (\dim M)/2$ are precisely those for which $b(G)\leq (\dim G)/2$. From Theorem~\ref{THM:observations_table_symmetric_spaces}, we know that the only Lie groups satisfying $b(G)\leq(\dim G)/2$ are the simple ones of rank $\leq 2$ and $\mathsf{F}_4$. The particular case $G=\su{2}\cong\sph^3$ is very remarkable, in that it yields a normal metric of $\Ric_2>0$.

\begin{remark}\label{REM:b_depends_on_pair_2}
The homogeneous space $(\sph^3\times\sph^3\times\sph^3)/\Delta \sph^3$ is the only simply connected one we know of that satisfies $b(-)=2$; compare to Item \eqref{item:Da} in Theorem~\ref{THM:observations_table_symmetric_spaces}. Observe that normal homogeneous metrics on $\sph^3\times\sph^3$ do not satisfy $\Ric_2>0$. Indeed, they are products of round metrics, and hence $b(\sph^3\times\sph^3)=\dim\sph^3 +1=4$ (cf.~Proposition~\ref{prop:products_homogeneous}).
\end{remark}

\subsection{Fat homogeneous bundles}\label{SS:fat_homogeneous_bundles}

After the compact rank one symmetric spaces and the Berger spaces $B^7$, $B^{13}$, the next examples of closed manifolds of $\sec>0$ were constructed by Wallach in 1972 \cite{Wa72}; namely the spaces $W^6\defeq\su{3}/\u{1}^2$, $W^{12}\defeq\sp{3}/\sp{1}^3$ and $W^{24}\defeq\mathsf{F}_4/\spin{8}$ now known as the \textit{Wallach flag manifolds}. Their common feature is that they arise as total spaces of \emph{fat homogeneous bundles} over compact rank one symmetric spaces. We refer the reader unfamiliar with fat homogeneous bundles to Section~\ref{SEC:fat_bundles} for definitions and references. Throughout this article we understand a metric on $G/H$ to be homogeneous if it is $G$-invariant, i.e.~our definition depends on the pair $(G,H)$.

Wallach's construction can be stated as follows: if $H<K<G$ is a nested triple of compact Lie groups such that the induced bundle $K/H\to G/H\to G/K$ is fat and $G/K$ is a rank one symmetric space, then $G/H$ admits a homogeneous metric of $\sec>0$ \cite[Proposition~4.3]{Zi07}. This approach was subsequently used by Aloff and Wallach in \cite{AW75} to construct homogeneous metrics of $\sec>0$ on the spaces $W_{p,q}^7$; see \eqref{eq:Wpq}. This family completes the list of simply connected homogeneous manifolds admitting a metric of $\sec>0$, up to diffeomorphism \cite{WZ18}. Observe that Wallach's construction results in metrics with ``more'' curvature than the corresponding normal metrics since, by Berger's classification, none of the spaces $W^6$, $W^{12}$, $W^{24}$, $W_{p,q}^7$ admit normal homogeneous metrics of $\sec>0$, with a subtle exception in the case of $W_{1,1}^7$; see Remark \ref{REM:b_depends_on_pair}.

Wallach's construction has been generalized to several other curvature conditions (such as almost positive curvature \cite{Wi02}, quasi-positive curvature \cite{Ta03,KT14} or strongly positive curvature \cite{BM18}). Here, we present the following generalization of Wallach's result:

\begin{theo}\label{THM:Wallach_generalization}
Let $K/H\to G/H\to G/K$ be a fat homogeneous bundle with $G/K$ a symmetric space. Then $G/H$ admits a homogeneous metric of $\Ric_k>0$ for $k= b(G/K)$.
\end{theo}

Fat homogeneous bundles are very rigid; indeed, a full classification of them was established by Bérard-Bergery \cite{Be75}. In Subsection~\ref{SS:fat_bundles_examples}, we examine fat homogeneous bundles whose base space $G/K$ is symmetric and satisfies $b(G/K)\leq (\dim G/K)/2$. By Theorem~\ref{THM:Wallach_generalization}, the corresponding total spaces $G/H$ admit homogeneous metrics of $\Ric_k>0$ for some $k< (\dim G/H)/2$. Theorem~\ref{THM:homogeneous_spaces_fat} provides a list of such total spaces $G/H$ and the $k$ for which $\Ric_k>0$. Here we only discuss in detail the necessary cases for the proof of Theorem~\ref{THM:main_theorem_existence_short}. The following triples induce fat homogeneous bundles:
\begin{equation}\label{EQ:fat_bundle_Aloff_Wallach}
\sss (\u{1}^{p,q}\times \u{n-1})  < \sss (\u{2}\times \u{n-1})< \su{n+1}, \qquad n\geq 2.
\end{equation}
The inclusions are induced by standard diagonal block embeddings, except for the first one, where $p,q$ are coprime integers with $pq>0$, and
\begin{equation}\label{EQ:circle_torus}
\u{1}^{p,q}\defeq\{ (z^p,z^q) : z\in \u{1}\}<\u{1}\times\u{1}.
\end{equation}
The corresponding bundles have the Grassmannian $\gr{2}{n+1}{\C}\defeq\su{n+1}/\sss (\u{2}\times \u{n-1})$ as base space, and the respective total spaces are the \emph{generalized Aloff-Wallach spaces}
\begin{equation}\label{eq:Wpq}
	W_{p,q}^{4n-1}\defeq \su{n+1}/\sss(\u{1}^{p,q}\times \u{n-1}).
\end{equation}
%\mathbb{P}_\CC T\CP^n &\defeq \su{n+1}/\sss(\u{1}\times \u{1}\times \u{n-1})\label{eq:PCTCP}
%, and , the projectivized tangent bundle of $\CP^n$:
Note that $W_{1,1}^{4n-1}$ equals the unit tangent bundle $T^1\CP^n$. Applying Theorem~\ref{THM:Wallach_generalization}, in combination with the value of $b(\gr{2}{n+1}{\C})$ from Table~\ref{table:b_symmetric_rank2} in Section~\ref{SEC:symmetric_spaces}, we obtain the first row of Item \eqref{item:Fa} in Theorem~\ref{THM:homogeneous_simple_version} below.

The last four rows of Item \eqref{item:Fa} in Theorem~\ref{THM:homogeneous_simple_version} below correspond to other infinite series of manifolds that arise as total spaces of fat homogeneous bundles (let $n\geq 2$): 
\begin{itemize}
\item the projectivized tangent bundle $\mathbb{P}_\CC T\CP^n\defeq \su{n+1}/\sss(\u{1}\times \u{1}\times \u{n-1})$ of $\CP^n$,
\item $\su{n+1}/(\su{2}\times\su{n-1})$, which can be described as the total space of a circle bundle over $\gr{2}{n+1}{\C}$ whose Euler class is a generator of $H^2(\gr{2}{n+1}{\C};\ZZ)= \ZZ$,
\item the projectivized tangent bundle $\mathbb{P}_\HH T\HP^n\defeq \sp{n+1}/(\sp{1}^2\times\sp{n-1})$ of $\HP^n$, and
\item the unit tangent bundle $T^1\sph^{n}\defeq \so{n+1}/\so{n-1}$ of the sphere $\sph^n$. %and its quotient the projectivized tangent bundle $\mathbb{P}_\RR T\RP^n\defeq \so{n+1}/\sss (\oo{1}^2\times\oo{n-1})$ of $\RP^n$.
\end{itemize}
Observe that $\mathbb{P}_\CC T\CP^2$ and $\mathbb{P}_\HH T\HP^2$ coincide with the Wallach spaces $W^6$ and $W^{12}$ (and moreover recall that the projectivized tangent bundle of the Cayley plane coincides with $W^{24}$). We refer to Subsection~\ref{SS:fat_bundles_examples} and Theorem~\ref{THM:homogeneous_spaces_fat} for the corresponding constructions and values, from where Item \eqref{item:Fb} below can be likewise extracted.

\begin{theo}\label{THM:homogeneous_simple_version}
The following hold:
\begin{enumerate}[\rm(a)]

\item Each homogeneous space $G/H$ in the following table carries a homogeneous metric of $\Ric_k>0$ for the corresponding value of $k$, where $n\geq 2$:\label{item:Fa}
\renewcommand{\arraystretch}{1.4}
\begin{center}
$
\begin{array}{c@{\quad}c@{\quad}c@{\quad}c} 
\hline
%G/H & \dim G/H & \ \  k,\ n\neq 3 &\ \ k,\ n=3 
\multirow{2}{*}{\vspace{1ex}$G/H$} & \multirow{2}{*}{\vspace{1ex}$\dim G/H$} &  \multicolumn{2}{c}{k}     \\[-1.5ex] %\cline{3-4}
&&{\scriptstyle(n\neq 3)} & {\scriptstyle(n= 3)}
\\  
\hline
W_{p,q}^{4n-1}  & 4n-1  &  2n-3 & 4   \\ %[0.5ex] 
\hline 
\mathbb{P}_\CC T\CP^n   & 4n-2  &  2n-3  &   4\\ %[0.5ex] 
\hline
\su{n+1}/(\su{2}\times\su{n-1}) & 4n-3 & 2n-3 &   4 \\ %[0.5ex] 
\hline
\mathbb{P}_\HH T\HP^n   & 8n-4  & 4n-7 & 6  \\ %[0.5ex] 
\hline 
T^1\sph^{n} & 2n-1  &  \multicolumn{2}{c}{n-1}  \\%[0.5ex] 
\hline
\end{array}
$
\vspace{1ex}
\end{center}

%\item The spaces $W_{p,q}^{4n-1}$ and $\mathbb{P}_\CC T\CP^n$ carry homogeneous metrics of $\Ric_{2n-3}>0$ if $n\geq 2$ and $n\neq 3$, and of $\Ric_4>0$ if $n=3$.

%\item The spaces $\mathsf{G}_2/\su{2}^\pm$ and $\mathsf{G}_2/\u{2}^\pm$ carry homogeneous metrics of $\Ric_{3}>0$.

%\item The following spaces $G/H$ carry homogeneous metrics of $\Ric_k>0$ for some $k\leq(\dim G/H)/3$: $W_{p,q}^{15}$, $\sp{4}/(\sp{1}^2\times \sp{2})$, $\sp{5}/(\sp{1}^2\times \sp{3})$, $\su{6}/\sss (\u{2}\times \sp{2})$, $\mathsf{G}_2/\su{2}^{\pm}$, $\mathsf{G}_2/\u{2}^{\pm}$, $\mathsf{E}_6/\spin{10}$ and $\so{10}/\su{5}$.

\item The following $G/H$ carry homogeneous metrics of $\Ric_k>0$ for some $k\leq(\dim G/H)/3$:

 \label{item:Fb}
$W_{p,q}^{15}$, $\su{6}/\sss (\u{2}\times \sp{2})$, $\mathsf{G}_2/\su{2}^{\pm}$, $\mathsf{G}_2/\u{2}^{\pm}$, $\mathsf{E}_6/\spin{10}$, $\so{10}/\su{5}$, $\mathbb{P}_\HH T\HP^3$ and $\mathbb{P}_\HH T\HP^4$. % (the projectivized tangent bundles of $\HP^3$ and $\HP^4$).
\end{enumerate}
\end{theo}

The cohomology ring of $W_{p,q}^{4n-1}$ can be computed using standard techniques, see Subsection~\ref{SS:fat_bundles_topology}. In particular, $W_{p,q}^{4n-1}$ is simply connected and $H^{2n}(W_{p,q}^{4n-1};\ZZ)$ is finite of order equal to $ \vert p^n+p^{n-1}q+\dots + pq^{n-1}+q^n\vert $; see Lemma~\ref{LEM:topology_Aloff_Wallach} or \cite[p.~119]{Wi02}. Hence by varying $p,q$ appropriately, we obtain infinitely many spaces of pairwise distinct homotopy type. This, along with the first row of Theorem~\ref{THM:homogeneous_simple_version}~\eqref{item:Fa}, implies Theorem \ref{THM:main_theorem_existence_short}. We remark that Wilking (resp.\ Tapp) constructed on each $W_{p,q}^{4n-1}$ with $pq<0$ (resp.\ $pq\neq 0$) inhomogeneous metrics of positive sectional curvature on an open dense subset \cite{Wi02} (resp.\ of non-negative sectional curvature and positive at a point \cite{Ta03}).
%
%The unit bundle $T^1\sph^{n}$, which can equivalently be written as $\oo{n+1}/\oo{n-1}$, induces a Riemannian covering over the projectivized tangent bundle $\mathbb{P}_\RR T\RP^n\defeq \oo{n+1}/(\oo{1}^2\times\oo{n-1})$ of $\RP^n$. The latter inherits a homogeneous metric of $\Ric_k>0$ with the same $k$ as $T^1\sph^{n}$. Recall that $\mathbb{P}_\RR T\RP^n$ is non-orientable if $n$ is odd since it fits into a smooth fiber bundle $\RP^{n-1}\to\mathbb{P}_\RR T\RP^n\to\RP^n$ whose fiber is non-orientable, thus it does not admit metrics of $\sec >0$ by Synge's Theorem. This proves Theorem~\ref{THM:main_theorem_quotients} for dimensions $\geq 9$ congruent to $1\,\Mod 4$. 
%Non-orientability of $\mathbb{P}_\RR T\RP^n$ for odd $n$ can be checked using the associated smooth fiber bundle $\RP^{n-1}\xrightarrow{i}\mathbb{P}_\RR T\RP^n\to\RP^n$ which induces the identity $i^*w(\mathbb{P}_\RR T\RP^n)=w(\RP^{n-1})$ between Stiefel-Whitney classes; with this at hand recall that the fiber $\RP^{n-1}$ is non-orientable if and only if $n-1$ is even and that a manifold is orientable if and only if its first Stiefel-Whitney class vanishes.

The unit tangent bundle $T^1\sph^n$ is especially interesting in the cases $n=3,7$, since the corresponding spheres $\sph^n$ are parallelizable and hence $T^1\sph^n$ is diffeomorphic to $\sph^{n-1}\times\sph^n$. Observe that the corresponding metrics of $\Ric_k>0$ with $k=2,6$ given by Theorem~\ref{THM:homogeneous_simple_version} improve the standard product metrics, which are only of $\Ric_k>0$ for $k=4,8$, respectively. The case $n=3$ is somehow exceptional since the metric given by the fat bundle construction is only of $\Ric_3>0$ (cf.~Theorem~\ref{THM:homogeneous_spaces_fat}). However,  $T^1\sph^3=\so{4}/\so{2}$ is $\sp{1}\times\sp{1}$-equivariantly diffeomorphic to the homogeneous space $(\sph^3\times\sph^3)/\Delta\sph^1$, which does admit a metric of $\Ric_2>0$, as it will be discussed in Remark~\ref{rem:products_spheres}; see~Remark~\ref{rem:T1S3} for details.

%{\color{red}The unit bundle $T^1\sph^n$ is especially interesting in the cases $n=3,7$, since the corresponding spheres $\sph^n$ are parallelizable and hence $T^1\sph^n$ is diffeomorphic to $\sph^{n-1}\times\sph^n$. Observe that the corresponding metrics of $\Ric_k>0$ with $k=2,6$ given by Theorem~\ref{THM:homogeneous_simple_version} improve the standard product metrics which are only of $\Ric_k>0$ for $k=4,8$, respectively. The case $n=3$ is somehow exceptional since the metric given by the fat bundle construction is only of $\Ric_3>0$ (cf.~Theorem~\ref{THM:homogeneous_spaces_fat}); however we will see in SECTION 5.2? that $T^1\sph^3$ is $\sp{1}\times\sp{1}$-equivariant to a certain homogeneous space which, as it shall be discussed in Subsection~\ref{SS:Ric2}, does admit a metric of $\Ric_2>0$.}

The spaces $\mathsf{G}_2/\su{2}^-$ and $\mathsf{G}_2/\u{2}^-$ appearing in Item \eqref{item:Fb} carry metrics of $\Ric_3>0$ by Theorem~\ref{THM:homogeneous_spaces_fat}, and are diffeomorphic to the unit tangent bundle $T^1\sph^6$ (for which the homogeneous metric given by Theorem~\ref{THM:homogeneous_simple_version} is ``only'' of $\Ric_5>0$) and to $\grp{2}{7}{\R}$, respectively \cite{KZ17}. The spaces $\su{3}/\so{3}$, $\grp{2}{5}{\R}$, $\grp{2}{7}{\R}$, and $\mathsf{G}_2/\so{4}$ are the only (simply connected) irreducible symmetric spaces of rank $>1$ we know of to admit metrics of $\Ric_3>0$; see Table~\ref{table:low_b_symmetric} in Section~\ref{SEC:symmetric_spaces}. 
It is an open question whether there exists an irreducible symmetric space of rank $>1$ which admits a metric of $\Ric_2>0$, much less of $\sec >0$ (see \cite{DN20} for the construction of a metric on $\grp{2}{7}{\R}$ of $\sec >0$ on an open dense subset).

\subsection{Isometric quotients and manifolds of $\mathrm{Ric}_2>0$}\label{SS:Ric2}

%Until the present work and \cite{AQZ20}, we are not aware of any systematic study to find metrics of $\Ric_k>0$, or to find the least $k$ for which well-known metrics (e.g.~normal homogeneous or biquotient metrics) satisfy $\Ric_k>0$. To the best of our knowledge, the only known source of examples was provided by Burkhard Wilking and shared to the third author in a personal communication. The starting point is: the Cheeger deformation of the product of round metrics on $\sph^3\times\sph^3$ via the diagonal action by $\Delta\sph^3 <\sph^3\times\sph^3$ is of $\Ric_2>0$. We refer the reader unfamiliar with Cheeger deformations to \cite[Section~2]{Zi07}. In this particular case, Wilking's metric on $\sph^3\times\sph^3$ is isometric to a normal homogeneous metric on $(\sph^3\times\sph^3\times\sph^3)/\Delta \sph^3$. His construction generalizes as follows to the product $G\times G$ for any compact semisimple Lie group $G$.

The metrics considered in the previous subsections have considerably large isometry groups, and hence their quotients constitute a large class of examples.
It is straightforward to see that quotient spaces of large enough dimension inherit metrics of positive intermediate Ricci curvature by O'Neill's curvature equations for Riemannian submersions; see Corollary~\ref{COR:Riem_sub}.
Also recall that Synge's Theorem states that if a closed manifold is orientable, not simply connected, and even-dimensional, or if it is non-orientable and odd-dimensional, then it cannot admit a metric of $\sec>0$.
So by taking appropriate quotients of previous examples, we will now establish Theorem~\ref{THM:main_theorem_quotients} using Synge's Theorem.

%The metrics considered in the previous subsections have considerably large isometry groups and hence one is allowed to take many isometric quotients. It is straightforward to see that these quotient spaces inherit metrics of positive intermediate Ricci curvature, by O'Neill's formula for Riemannian submersions, see Corollary~\ref{COR:Riem_sub}. In this subsection we consider some of these quotients, mostly non-simply connected ones. In particular, the discussion below will prove Theorem~\ref{THM:main_theorem_quotients}. We will make constant use of Synge's Theorem, stating that any closed, orientable, non-simply connected manifold of even dimension, as well as any closed, non-orientable manifold of odd dimension, cannot admit a metric of $\sec >0$.

\begin{proof}[Proof of Theorem~\ref{THM:main_theorem_quotients}]
	For any $n\geq 3$, the Grassmannian $\gr{2}{n+2}{\R}=\so{n+2}/\sss(\oo{2}\times\oo{n})$ of (non-oriented) $2$-planes in $\R^{n+2}$ is an irreducible symmetric space of dimension $2n$ and rank $2$. It is a $\mathbb{Z}_2$-quotient of the oriented $2$-plane Grassmannian $\grp{2}{n+2}{\R}$. By Theorem~\ref{THM:observations_table_symmetric_spaces} (cf.~Table~\ref{table:b_symmetric_rank2}), its symmetric metric is of $\Ric_{k}>0$, for $k=n =(\dim \gr{2}{n+2}{\R})/2$, whenever $n\geq 3$. It is well known that  $\gr{2}{n+2}{\R}$ is orientable if and only if $n$ is even (this follows, for example, from a characterization of the orientability of homogeneous spaces~\cite[p.~275]{Ma72}). Thus, for each even $n\geq 4$, the orientable but non-simply connected manifold $\gr{2}{n+2}{\R}$ cannot admit a Riemannian metric of $\sec>0$, by Synge's Theorem. This shows Theorem~\ref{THM:main_theorem_quotients} for dimensions~$\geq 8$ that are multiples~of~$4$.
	
	The unit tangent bundle $T^1\sph^{n}=\so{n+1}/\so{n-1}$ induces a Riemannian covering over the projectivized tangent bundle $\mathbb{P}_\RR T\RP^n\defeq \so{n+1}/\mathsf{S}(\oo{1}^2\times\oo{n-1})$ of $\RP^n$. The latter inherits a homogeneous metric of $\Ric_k>0$ with the same $k$ as $T^1\sph^{n}$, namely $k=n-1<(\dim\mathbb{P}_\RR T\RP^n)/2$. Recall that $\mathbb{P}_\RR T\RP^n$ is non-orientable if $n$ is odd since it fits into a smooth fiber bundle $\RP^{n-1}\to\mathbb{P}_\RR T\RP^n\to\RP^n$ whose fiber is non-orientable. Thus $\mathbb{P}_\RR T\RP^n$ does not admit metrics of $\sec >0$ by Synge's Theorem. This proves Theorem~\ref{THM:main_theorem_quotients} for dimensions $\geq 5$ congruent to $1\,\Mod 4$. 
	%This proves Theorem~\ref{THM:main_theorem_quotients} for dimensions $\geq 9$ congruent to $1\,\Mod 4$. 
	
	The case of dimension $4$ in Theorem~\ref{THM:main_theorem_quotients} follows from the metric of $\Ric_2 >0$ on $\so{3}\times\so{3}$  resulting from Theorem~\ref{THM:Wilkings_products}. The quotient of $\so{3}\times\so{3}$ under the left action by the product subgroup $\oo{2}\times\oo{2}$ is diffeomorphic to $\RP^2\times\RP^2$. Then $\RP^2\times\RP^2$ inherits a metric of $\Ric_2>0$ by Corollary~\ref{COR:Riem_sub}, but cannot admit a metric of $\sec >0$.
\end{proof}

\begin{remark}\label{rem:products_spheres}
	The Wilking metric of $\Ric_2>0$ on $\sph^3\times\sph^3$ is left-invariant and right $\Delta\sph^3$-invariant, see Theorem~\ref{THM:Wilkings_products}. By Corollary~\ref{COR:Riem_sub} it induces metrics of $\Ric_2>0$ on $\sph^2\times\sph^2$ and $\sph^2\times\sph^3$, by taking the quotients of the left action of a maximal torus of $\sph^3\times\sph^3$, and of the right $\Delta \sph^1$-action, respectively. 
	The manifolds $\sph^2\times\sph^2$, $\sph^2\times\sph^3$, and $\sph^3\times\sph^3$ are the only simply connected ones we know of to admit metrics of $\Ric_2>0$, besides those of $\sec >0$.
\end{remark}

The approach followed above can be extended to other homogeneous spaces of low dimensions to generate more examples of manifolds admitting metrics of $\Ric_k>0$ but not $\sec>0$. Specifically, one can consider bi-invariant metrics on a compact simple Lie group~$G$, as well as the homogeneous metrics on $G^2$ given in Theorem~\ref{THM:Wilkings_products}. Both have large groups of isometries acting freely, and thus give rise to many isometric quotients. We can, for instance, consider subgroups $H$ of $G$ or $G^2$ of the form $H=S\cdot \Gamma$, where $S$ is a closed connected subgroup with $\rank S < \rank G$, and $\Gamma$ is a finite subgroup of a torus $\mathbb{T}\subset G$ orthogonal to $S$. If we restrict to quotients $G/H$ or $G^2/H$ of even dimension, and take subgroups $\Gamma$ such that $H$ has at least $3$ connected components, then the quotients cannot admit metrics of $\sec>0$ by Synge's Theorem. In order to obtain quotients of $\Ric_k>0$ where $k$ is at most half of their dimensions, by Theorem~\ref{THM:observations_table_symmetric_spaces} we have to restrict to simple Lie groups of rank $\leq 2$ and to $\mathsf{F}_4$, and their squares. Thus, for example, taking $S$ a (possibly trivial) even-dimensional torus in the groups $G=\su{3}$, $\sp{2}$, $\mathsf{G}_2$, $\mathsf{F}_4$ with bi-invariant metrics, we obtain examples in dimensions $8$, $10$, $14$, $50$ and $52$. We can argue similarly with squares $G^2$ endowed with the metrics from Theorem~\ref{THM:Wilkings_products}, thus obtaining examples in many even dimensions $\leq 104=2\dim \mathsf{F}_4$. The case of the group $\mathsf{F}_4\times\mathsf{F}_4$ is rather interesting, in that it allows to show that each finite subgroup of the $8$-torus can be realized as the fundamental group of a closed $n$-manifold of $\Ric_k>0$ for some $k<n/2$. Indeed, $\mathsf{F}_4\times\mathsf{F}_4$ carries a homogeneous metric of $\Ric_{44}>0$ (by Theorem~\ref{THM:Wilkings_products} and Table~\ref{table:b_symmetric}). Taking the quotient under any finite subgroup $\Gamma$ of a maximal torus $\mathbb{T}^8$ we obtain a manifold of dimension $104$ with fundamental group equal to $\Gamma$ and a metric~of~$\Ric_{44}>0$.

We conclude this subsection with an observation regarding Riemannian submersions and the Wilking metric of $\Ric_2>0$ on $\mathbb{S}^3\times\mathbb{S}^3$ discussed above Theorem~\ref{THM:Wilkings_products}. The right action by $\Delta \sph^3$ on $\sph^3\times\sph^3$ induces a metric on $(\sph^3\times\sph^3)/\Delta \sph^3$ that inherits the left action by $\sph^3\times\sph^3$. As such it is a homogeneous metric (i.e. $\sph^3\times\sph^3$-invariant), which must be normal since $(\sph^3\times\sph^3)/\Delta \sph^3$ is an irreducible symmetric pair. Under the natural diffeomorphism $(\sph^3\times\sph^3)/\Delta \sph^3\cong \sph^3$ (see Remark~\ref{REM:bG_equals_bGxG_DeltaG}) the metric becomes bi-invariant on $\sph^3$, which is round and hence submerges onto a metric $\bar g$ of constant curvature on $\sph^2$. Altogether, we have a Riemannian submersion from Wilking's metric $g$:
\begin{equation}\label{EQ:submersion_S3S3_S2}
(\sph^3\times\sph^3,g) \to (\sph^2,\bar g)\quad \text{ with } \Ric_2(g)>0.
\end{equation}
	Riemannian submersions between closed manifolds satisfying certain curvature bounds are challenging to construct, and it is expected that there might be topological obstructions for their existence. 
	In this context, the Petersen-Wilhelm Conjecture suggests that any Riemannian submersion $M\to B$ defined on a positively curved manifold $M$ with fibers $F$ must have $\dim F <\dim B$; see \cite{AK15,GG16,Sp18,GR18} for further information. 
	Although this conjecture is still open, it is natural to ask whether one should expect an analogous fiber dimension bound for Riemannian submersions defined on manifolds with $\Ric_k>0$.
	More precisely, for each $k\geq 1$ one could ask if there exists a positive number $f(k)$ such that if $\Ric_k (M) > 0$, then $\dim F < f(k) \dim B$. 
	In this language, the Petersen-Wilhelm Conjecture suggests that $f(1)$ exists and is equal to $1$. The submersion in \eqref{EQ:submersion_S3S3_S2} shows that if $f(2)$ exists, then its value is strictly greater than $2$.

\subsection*{Structure of the article}

In Section~\ref{SEC:normal_homogeneous_spaces} we consider normal homogeneous spaces and give the proofs of Theorems~\ref{THM:extension_Nash_Ber} and~\ref{THM:Wilkings_products}. Section~\ref{SEC:symmetric_spaces} is devoted to the  computation of $b(G/K)$ for any symmetric space $G/K$ of compact type, and contains the proof of Theorem~\ref{THM:observations_table_symmetric_spaces}. In Section~\ref{SEC:fat_bundles}, we review fat homogeneous bundles, we provide the proof of Theorem~\ref{THM:Wallach_generalization}, we discuss various families of fat homogeneous bundles, from where Theorem~\ref{THM:homogeneous_simple_version} follows, and we prove the topological claims in Theorem~\ref{THM:main_theorem_existence_short}. Finally, we include Table~\ref{table:b_symmetric} with the values of $b(G/K)$ for all irreducible compact symmetric spaces $G/K$ at the end of the article.

\subsection*{Acknowledgments}

We are grateful to Fernando Galaz-Garc\'ia, Luis Guijarro, Lorenz Schwachh\"ofer and Fred Wilhelm for helpful correspondence.
We also thank Anand Dessai and Lee Kennard for useful comments on a preliminary version of this article.
In particular, we thank Lee Kennard for helping us detect an issue with a previous version of this article.

%% file: preliminaries.tex
\section{Background \& Context}\label{S:preliminaries}

	In this section, we give a general overview of basic concepts related to intermediate Ricci curvature. 
	Our goal is to compensate for the fact that there is limited existing literature on intermediate Ricci curvature.
	At the end of this section, we also include a discussion of other different intermediate curvature conditions that others have studied.

\subsection{Definition and basic constructions}
	
	Here, we provide a more general definition of $\Ric_k$, namely as a map defined on flags that, when it is required to be positive, restricts to the definition given at the beginning of this article.
	To that end, let $\mathrm{Fl}(1,k+1;TM)$ denote the \textit{partial flag bundle} consisting of signature-$(1,k+1)$ flags tangent to $M$. 
	In other words, elements of $\mathrm{Fl}(1,k+1;TM)$ are pairs $(\mathcal{V}^1,\mathcal{V}^{k+1})$ such that $\mathcal{V}^1$ is a $1$-dimensional subspace of a $(k+1)$-dimensional subspace $\mathcal{V}^{k+1}$ of a tangent space $T_pM$ for some $p\in M$.
	Given a Riemannian metric on $M$, let $\nabla$ and $R$ denote the associated Levi-Civita connection and curvature tensor, respectively.
	For any vector $x\in T_pM$, the \textit{directional curvature operator} or \textit{Jacobi operator} $R_x\colon T_pM\to T_pM$ is then defined by
	\[
		R_x(y) \defeq R(y,x)x = \nabla_Y\nabla_X X - \nabla_X \nabla_Y X - \nabla_{[Y,X]}X,
	\]
	where $X$ and $Y$ are any extensions of $x$ and $y$ to smooth vector fields, respectively.
	
	\begin{definition}
		Suppose $(M,g)$ is a Riemannian manifold and $k\in\{1,\dots,\dim M-1\}$. 
		Given $p\in M$, a $(k+1)$-dimensional subspace $\mathcal{V}^{k+1}\subset T_pM$, and a vector $x\in \mathcal{V}^{k+1}$, let $R_x^{\mathcal{V}^{k+1}}$ denote the directional curvature operator restricted to $\mathcal{V}^{k+1}$ composed with the orthogonal projection onto $\mathcal{V}^{k+1}$. 
		Then $\Ric_k\colon \mathrm{Fl}(1,k+1;TM)\to\mathbb{R}$ is defined by
		\[
			\Ric_k(\mathcal{V}^1,\mathcal{V}^{k+1})\defeq\trace\big(R_x^{\mathcal{V}^{k+1}}\big)=\sum_{i=1}^k\sec(x,e_i),
		\]
		where $x$ is any unit vector in the line $\mathcal{V}^1$, and $e_1,\dots,e_k$ is any orthonormal basis for the orthogonal complement of $\mathcal{V}^1$ in $\mathcal{V}^{k+1}$. 
	\end{definition}
	
	Given a signature-$(1,2,n)$ flag $\mathcal{V}^1\subset \mathcal{V}^2\subset \mathcal{V}^n$ tangent to an $n$-manifold, notice that $\Ric_1(\mathcal{V}^1,\mathcal{V}^{2})$ is the sectional curvature of the $2$-plane $\mathcal{V}^{2}$, and $\Ric_{n-1}(\mathcal{V}^1,\mathcal{V}^n)$ is the Ricci curvature $\Ric(x,x)$ for either choice of unit vector $x\in\mathcal{V}^1$. 
	Furthermore, notice that this notion of $\Ric_k$ is well-defined because the value of $\trace(R_x^{\mathcal{V}^{k+1}})=\sum_{i=1}^k\sec(x,e_i)$ is independent of the choice of unit vector $x\in \mathcal{V}^1\subset \mathcal{V}^{k+1}$ or orthonormal vectors $e_1,\dots,e_k\in\mathcal{V}^{k+1}\cap (\mathcal{V}^1)^\perp$.
	
	Riemannian products give us an elementary way of constructing manifolds of $\Ric_k>0$ from known examples. 
	%Given Riemannian manifolds $(M_1,g_1)$ and $(M_2,g_2)$, their Riemannian product is the Riemannian manifold $(M_1\times M_2,g_1+g_2)$, where $(g_1+g_2)((x_1,x_2),(y_1,y_2)) = g_1(x_1,y_1) + g_2(x_2,y_2)$ for all vectors $x_i,y_i$ tangent to $M_i$.
	%Then the Levi-Civita connection associated to $g_1+g_2$ splits into the Levi-Civita connections associated to $g_1$ and $g_2$ on each factor. %, see, for example, \cite[Chapter 6 Exercise 1]{docarmo}.
	Since the Levi-Civita connection of a Riemannian product splits as the sum of the Levi-Civita of both factors, it is straightforward to check the following:
	
	\begin{proposition}\label{PROP:products}
		Suppose $(M_1,g_1)$ and $(M_2,g_2)$ are Riemannian manifolds such that $\Ric_{k_i}(M_i,g_i)>0$ for some $k_i\in\{1,\dots,\dim(M_i)-1\}$. 
		Then their Riemannian product $(M_1\times M_2,g_1+g_2)$ is of $\Ric_k>0$ for $k=\max\{k_1+\dim(M_2),k_2+\dim(M_1)\}$.
	\end{proposition}
	
	Riemannian submersions give another source of examples of manifolds with $\Ric_k>0$. An important subclass in this article is given by quotients of free, isometric actions by compact Lie groups. 
	%Recall that a smooth submersion between Riemannian manifolds $\pi\colon(M,g_M)\to (B,g_B)$ is called a Riemannian submersion if $g_B(d\pi(x),d\pi(y)) = g_M(x,y)$ for vectors $x,y$ in the horizontal distribution $(\ker d\pi)^\perp$.
	%One common source of examples is given by quotients of free, isometric actions by compact Lie groups.
	%O'Neill determined fundamental equations that relate the curvature tensors of the domain $(M,g_M)$, the base $(B,g_B)$, and the fibers of the submersion \cite{oneill}.
%	In particular, it follows from O'Neill's horizontal curvature equation that $\sec_B(d\pi(x),d\pi(y))\geq \sec_M(x,y)$ for all orthonormal vectors $x,y$ in the horizontal distribution.
	It follows from O'Neill's horizontal curvature equation that a Riemannian submersion $\pi\colon (M,g_M)\to (B,g_B)$ satisfies $\sec_B(d\pi(x),d\pi(y))\geq \sec_M(x,y)$ for all orthonormal vectors $x,y$ in the horizontal distribution $(\ker d\pi)^\perp$; see~\cite{GW09} for further information.
	Thus, the following is an immediate consequence:
	
	\begin{corollary}\label{COR:Riem_sub}
		Suppose $\pi\colon(M,g_M)\to (B,g_B)$ is a Riemannian submersion. If $\Ric_k>0$ on the horizontal distribution in $M$ for some $k\in\{1,\dots,\dim(B)-1\}$, then $\Ric_k(B,g_B)>0$.
	\end{corollary}

\subsection{Related results and curvature conditions}\label{SS:related_work}
	
	We close this section by highlighting some existing work on manifolds of $\Ric_k>0$ and related curvature conditions. Various classical results for manifolds with lower bounds on sectional or Ricci curvature have been extended appropriately to the context of $\Ric_k>0$ by several authors. Synge's Theorem was generalized by Wilhelm \cite{Wi97}, Heintze-Karcher's Inequality by Chahine \cite{Ch18}, and Gromoll-Meyer's Theorem and Cheeger-Gromoll's Soul Theorem on open manifolds by Shen \cite{Sh93}. Moreover, various comparison results have been established by Guijarro and Wilhelm in the series of papers \cite{GW18,GW20,GW}, while the third named author is carrying out a systematic study of manifolds of $\Ric_k>0$ under the presence of isometric actions by large Lie groups \cite{Mo19a,Mo20}. 
	
	Now given a unit vector $x\in T_pM$, let $x^\perp$ denote the orthogonal complement of $\mathrm{span}\{x\}$ in $T_pM$. 
	Then the restriction of the directional curvature operator $R_x$ to $x^\perp$ is a self-adjoint endomorphism.
	Verdiani and Ziller studied manifolds for which the sum of the $k$ smallest eigenvalues of $R_x|_{x^\perp}$ is positive or non-negative in \cite{VZ14}.
	One advantage of defining $\Ric_k$ as we have in this section is that one can easily see this condition is equivalent to $\Ric_k$ being positive or non-negative as a  consequence of the Min-Max Theorem applied to $R_x|_{x^\perp}$:
	
	\begin{proposition}
		On any Riemannian manifold, $\Ric_k$ is positive (resp.\ non-negative) if and only if the sum of the $k$ smallest eigenvalues of the operator $R_x|_{x^\perp}$ (counted with multiplicity) is positive (resp.\ non-negative) for all unit vectors $x$.
	\end{proposition}
	
	On the other hand, there has been active research on certain curvature conditions that interpolate between the classical positive Ricci and positive scalar curvature. These conditions have often been called $k$-Ricci curvature in the literature, which should not be mistaken with the curvature conditions under consideration in this work. We refer to the recent preprints \cite{CW20,WW20} for further information on the topic.
			
	Beyond this condition, many researchers have investigated other notions of partially positive curvature.
	For example, Gromov proposed the study of the so-called positive $p$-curvature, which interpolates between positive sectional curvature when $p=n-2$ and positive scalar curvature when $p=0$, where $n$ is the dimension of the manifold. For more information, see~\cite{La97} or \cite[Section~2.2]{Ko20}.
	For a second example, a Riemannian manifold $(M,g)$ is said to have $p$-positive curvature operator if the sum of the $p$ smallest eigenvalues of the curvature operator $\g{R}\colon\Lambda^2TM\to\Lambda^2TM$ is positive.
	B\"ohm and Wilking proved that a metric with $2$-positive curvature operator evolves through Ricci flow to a metric of constant positive sectional curvature, thus showing that such spaces are spherical space forms~\cite{BW08}.
	
%	{\color{red}SAY SOMETHING ABOUT OTHER CONDITIONS? I'll try to write a super short summary like the one I wrote in some email to you.}
%	
%	{\color{blue} 
%		My rough proposal for ending this section:
%		
%		Beyond this condition, many researchers have studied other notions of partially positive curvature.
%		For example, Gromov introduced positive $p$-curvature in [CITE GROMOV], which interpolates between positive sectional curvature when $p$ is two less than the dimension and positive scalar curvature when $p=0$.
%		For more information, see [CITE ON THE SPACE OF RIEMANNIAN METRICS SATISFYING SURGERY STABLE CURVATURE CONDITIONS - JAN-BERNHARD KORDAß - Section 2.2? Something better?]
%		For a second example, a Riemannian manifold $(M,g)$ is said to have $p$-positive curvature operator if the sum of the $p$ smallest eigenvalues of the curvature operator $\g{R}:\Lambda^2TM\to\Lambda^2TM$ is positive.
%		B\"ohm and Wilking proved that a metric with $2$-positive curvature operator evolves through Ricci flow to a metric of constant positive sectional curvature, thus showing that such spaces are spherical space forms [CITE BOHM WILKING].
%	}

%% file: homogeneous.tex
\section{Normal Homogeneous Spaces}\label{SEC:normal_homogeneous_spaces}
This section is devoted to the investigation of the positive intermediate Ricci curvature condition in the context of normal homogeneous spaces. In particular, in Subsection~\ref{subsec:normal} we derive various general properties of the smallest $k$ for which normal homogeneous spaces satisfy $\Ric_k>0$. In Subsection~\ref{subsec:thB} we briefly discuss Nash-Berestovskii's result on homogeneous spaces with positive Ricci curvature, and prove  Theorem~\ref{THM:extension_Nash_Ber}. In Subsection~\ref{subsec:thD} we consider certain homogeneous spaces with diagonal isotropy groups, and prove Theorem~\ref{THM:Wilkings_products}. Finally, as an interesting example, in Subsection~\ref{SEC:Witte_spaces} we discuss some curvature and topological properties of certain circle bundles over a product of two complex projective spaces.

\subsection{Positive intermediate Ricci curvature on normal homogeneous spaces}\label{subsec:normal}
Let $G$ be a compact connected Lie group and $H$ a closed subgroup of $G$, and denote their Lie algebras by $\g{g}$ and $\g{h}$ respectively. Endow the homogeneous space $G/H$ with a normal homogeneous Riemannian metric. Let $\g{p}$ be the orthogonal complement of $\g{h}$ in $\g{g}$ with respect to the $\Ad(G)$-invariant inner product on $\g{g}$ that defines the normal homogeneous metric on $G/H$. 

For each non-zero $x\in \g{p}$, let us define $b(G/H,x)$ as the smallest integer $b$ for which the following holds: for every set of orthonormal vectors $y_1,\dots, y_b\in\g{p}$ orthogonal to $x$, there is some $i\in\{1,\dots,b\}$ such that $[x,y_i]\neq 0$. Then we define
\[
b(G/H)=\max_{x\in\g{p}\setminus\{0\}} b(G/H,x).
\]
Note that this definition of $b(G/H)$ agrees with the definition given in Subsection~\ref{SS:normal_homogeneous_symmetric}. The equivalence immediately follows from the well-known curvature formula for normal homogeneous metrics; see e.g.\ \cite[Corollary 2.4.1]{GW09}:
$$
\sec (x,y)=\frac{1}{4}|[x,y]_\g{p}|^2 + |[x,y]_\g{h}|^2, \quad \text{for orthonormal } x,y\in\g{p}.
$$
As mentioned in Subsection~\ref{SS:normal_homogeneous_symmetric}, $b(G/H)$ may depend on the actual pair $(G,H)$ under consideration; see Remarks~\ref{REM:b_depends_on_pair} and \ref{REM:b_depends_on_pair_2}. However, as we will see in Corollary~\ref{cor:independence}, once the pair $(G,H)$ is fixed, the number $b(G/H)$ does not depend on the normal homogeneous metric chosen on $G/H$. 

The next proposition gives an upper bound of $b(G/H,x)$ by the codimension in $\g{p}$ of the orbit $\Ad(H)x$. Here, $\Ad(H)x$ is the orbit through $x$ of the restriction of the adjoint representation of $G$ to $H$. Observe that $\Ad(H)x\subset\g{p}$ can be identified with the orbit of the isotropy representation of $G/H$ through the tangent vector $x\in\g{p}\cong T_{eH}(G/H)$. The bound we prove is sharp when $[\g{p},\g{p}]\subset\g{h}$, i.e.\ when $(G,H)$ is a symmetric pair. More precisely,
\[
b(G/H)\leq \max_{x\in \g{p}\setminus\{0\}}\codim_\g{p} \Ad(H) x,
\]
with equality if $(G,H)$ is a symmetric pair. Of course, the maximum can be taken over the unit sphere of $\g{p}$.

\begin{proposition}\label{prop:centralizer}
	Let $G/H$ be a normal homogeneous space, and $\g{g}=\g{h}\oplus\g{p}$ the corresponding reductive decomposition. Then $b(G/H,x)=\dim Z_\g{p}(x)$, where $Z_\g{p}(x)=\{y\in\g{p}:[x,y]=0\}$ is the centralizer of $x$ in $\g{p}$, and hence
	\[
	b(G/H)=\max_{x\in\g{p}\setminus\{0\}} \dim Z_\g{p}(x).
	\]	
	Moreover, $Z_\g{p}(x)$ is a linear subspace of $\nu_x(\Ad(H)x)$, the normal space at $x$ to the orbit $\Ad(H)x\subset \g{p}$. 
\end{proposition}
\begin{proof}
	Fix $x\in\g{p}\setminus\{0\}$. If there is a set of orthonormal vectors $y_1,\dots, y_b\in\g{p}$ perpendicular to $x$ and such that $[x,y_i]=0$ for all $i$, then $\spann\{x,y_1,\dots, y_b\}\subset Z_\g{p}(x)$, and hence $b<\dim Z_\g{p}(x)$. Thus $b(G/H,x)\geq \dim Z_\g{p}(x)$. Consider now a set of orthonormal vectors $y_1,\dots, y_{\dim Z_\g{p}(x)}\in\g{p}$ perpendicular to $x$. If $[x,y_i]=0$ for all $i$, then $\spann\{x,y_1,\dots, y_{\dim Z_\g{p}(x)}\}\subset Z_\g{p}(x)$, whence $\dim Z_\g{p}(x)+1\leq \dim Z_\g{p}(x)$, a contradiction. Thus there must exist $y_i$ such that $[x,y_i]\neq 0$. Therefore, $b(G/H,x)\leq \dim Z_\g{p}(x)$, which yields the desired equality.
	
	The tangent space to the orbit $\Ad(H)x$ at $x$ is given by $[\g{h},x]$. Let $y\in \g{p}$. By the $\Ad(G)$-invariance of the inner product on $\g{g}$, we have $\langle [x,y],T\rangle = \langle y,[T,x]\rangle$ for any $T\in\g{h}$. Thus, $[x,y]_\g{h}=0$ (where the subscript denotes orthogonal projection) if and only if $y$ is orthogonal to the $\Ad(H)$-orbit through $x$ at $x$, i.e.\ $y\in\nu_x(\Ad(H)x)$. Similarly, for any $z\in \g{p}$ we have $\langle [x,y],z\rangle=\langle y, [z,x]\rangle$. Therefore, $[x,y]_\g{p}=0$ if and only if $y$ is orthogonal to $[\g{p},x]$. All in all, $[x,y]=0$ if and only if $y\in \nu_x(\Ad(H)x)\cap [\g{p},x]^\perp$, where $[\g{p},x]^\perp$ denotes the orthogonal complement of $[\g{p},x]$ in $\g{g}$. This shows that $Z_\g{p}(x)=\nu_x(\Ad(H)x)\cap [\g{p},x]^\perp\subset\nu_x(\Ad(H)x)$. 
\end{proof}

\begin{remark}\label{REM:bounds_nested_inclusions}
	An immediate consequence (which also follows by a standard Riemannian submersion argument) is that for a nested triple $H<K<G$ of compact Lie groups, we have $b(G/K)\leq b(G/H)$. Denote by $\mathfrak{h}\subset\mathfrak{k}\subset\mathfrak{g}$ their Lie algebras and fix a bi-invariant metric on $G$. For any $x\in\g{k}^\perp$ we have $Z_{\g{k}^\perp}(x)\subset Z_{\g{h}^\perp}(x)$, which implies $b(G/K,x)\leq b(G/H,x)$, and thus 
	\[
	b(G/K)=\max_{x\in\g{k}^\perp\setminus\{0\}} b(G/K,x)\leq \max_{x\in\g{k}^\perp\setminus\{0\}} b(G/H,x)\leq \max_{x\in\g{h}^\perp\setminus\{0\}} b(G/H,x) = b(G/H).
	\]
\end{remark}

\begin{corollary}\label{cor:independence}
	The number $b(G/H)$ is independent of the normal homogeneous Riemannian metric on $G/H$.
\end{corollary}

\begin{remark}
Corollary~\ref{cor:independence} is obvious if $G$ is simple, since all bi-invariant metrics on $G$ (and hence all normal metrics on $G/H$) are isometric up to scaling by Schur's lemma. Likewise, it clearly holds if $G/H$ is isotropy irreducible (cf.~\cite[Proposition~7.91]{Besse}).%, as all homogeneous metrics on $G/H$ are the standard one (i.e. the quotient of the bi-invariant metric on $G$ given by its Killing form) up to scaling \cite[Proposition~7.91]{Besse}.
\end{remark}

\begin{proof}[Proof of Corollary~\ref{cor:independence}]
	Let $\g{g}=\bigoplus_{i=0}^k\g{g}_i$, where $\g{g}_0$ is the abelian factor, and $\g{g}_i$, $i\geq 1$, are simple ideals. Fix an $\Ad(G)$-invariant inner product $\llangle\cdot,\cdot\rrangle$ on $\g{g}$. By Schur's lemma, any other $\Ad(G)$-invariant inner product $\langle\cdot,\cdot\rangle$ on $\g{g}$ satisfies $\langle\cdot,\cdot\rangle\rvert_{\g{g}_i\times\g{g}_i}=s_i \llangle\cdot,\cdot\rrangle$, for some $s_i>0$ and each $i\geq1$, and $\langle\cdot,\cdot\rangle\rvert_{\g{g}_0\times\g{g}_0}=\llangle A \,\cdot,\cdot\rrangle$, for some linear automorphism $A$ of $\g{g}_0$. Since each $\g{g}_i$, $i\geq 1$, is simple, we have $[\g{g}_i,\g{g}_i]=\g{g}_i$, which together with the $\Ad(G)$-invariance implies $\langle\cdot,\cdot\rangle\rvert_{\g{g}_i\times\g{g}_j}=0$ for each $i,j$, $i\neq j$.  Let $\g{m}$ (respectively\ $\g{p}$) be the orthogonal complement of $\g{h}$ in $\g{g}$ with respect to $\llangle\cdot,\cdot\rrangle$ (resp.\ with respect to $\langle\cdot,\cdot\rangle$). Consider the linear isomorphism $\varphi\colon \g{m}\to \g{p}$ given by $\varphi(\sum_i x_i)=A^{-1}x_0+\sum_{i\geq 1} s_i^{-1}x_i$, where $x_i\in\g{g}_i$. It is well defined, since for any $T=\sum_i T_i\in\g{h}$, $T_i\in\g{g}_i$, we have $\langle \varphi(\sum_i x_i), T\rangle=\langle A^{-1}x_0+\sum_{i\geq 1} s_i^{-1}x_i, \sum_i T_i\rangle=\llangle\sum_ix_i,T\rrangle=0$.
	
	%\marginpar{\tiny Triple check the proof of this result... {\color{red} I did various times, I believe it's correct. Let's see what Lawrence says.}}
	
	Let $x=\sum_i x_i$, $y=\sum_i y_i\in\g{m}$, $x_i$, $y_i\in\g{g}_i$. Then $y\in Z_\g{m}(x)$ if and only if $0=[y,x]=\sum_i[y_i,x_i]$, which is equivalent to $[y_i,x_i]=0$ for all $i$. Similarly, $\varphi(y)\in Z_\g{p}(\varphi(x))$ if and only if $0=[A^{-1}y_0+\sum_{i\geq 1} s_i^{-1}y_i,A^{-1}x_0+\sum_{i\geq 1} s_i^{-1}x_i]=\sum_{i\geq 1} s_i^{-2} [y_i,x_i]$, which again is equivalent to $[y_i,x_i]=0$ for all $i$. Therefore, $y\in Z_\g{m}(x)$ if and only if $\varphi(y)\in Z_\g{p}(\varphi(x))$. Since $\varphi$ is an isomorphism, it follows that $\dim Z_\g{m}(x)=\dim Z_\g{p}(\varphi(x))$, and then by Proposition~\ref{prop:centralizer}, we get that $b(G/H,x)=b(G/H,\varphi(x))$. Therefore, $b(G/H)$ does not depend on the choice of $\Ad(G)$-invariant inner product on $\g{g}$.
\end{proof}

For normal homogeneous spaces that are products, one can easily prove the following (cf.~Proposition~\ref{PROP:products}):
\begin{proposition}\label{prop:products_homogeneous}
	Let $G/H=G_1/H_1\times \dots \times G_s/H_s$ be a normal homogeneous space. Then $b(G/H)=\max\{b(G_j/H_j)+\dim G/H-\dim G_j/H_j:j=1,\dots,s\}$.
\end{proposition}

\subsection{The case $b(G/H)=\dim G/H-1$ and the proof of Theorem~\ref{THM:extension_Nash_Ber}}\label{subsec:thB}
Nash~\cite{Na79} and Berestovskii~\cite{Be95} proved that for a compact Lie group $G$ and a closed subgroup $H$ of $G$, the following are equivalent: 
\begin{enumerate}[\rm(i)]
	\item the fundamental group $\pi_1(G/H)$ is finite,
	\item there is a compact semisimple Lie group acting transitively on $G/H$,
	\item $G/H$ admits a metric of positive Ricci curvature,
	\item any normal metric on $G/H$ has positive Ricci curvature (i.e.\ $b(G/H)\leq \dim G/H-1$).
\end{enumerate} 
%In these conditions, there is a compact semisimple Lie group acting transitively on $G/H$, as explained in~\cite[proof of Proposition~3.4]{Nash}. 
Since we are interested in spaces admitting metrics of positive Ricci curvature, we will assume from now on that the normal homogeneous space $G/H$ is such that $G$ is semisimple.

\begin{remark}\label{REM:b_coverings}
Suppose $G$ is semisimple and hence $G/H$ has finite fundamental group. Then its universal covering space carries a natural homogeneous structure which can be described as $\widetilde{G}/\bar{H}$, where $\widetilde{G}$ is the universal cover of $G$ and $\bar{H}$ is the connected Lie subgroup of $\widetilde{G}$ with Lie algebra $\g{h}$. Since $b(-)$ depends on the corresponding Lie algebras and these remain unchanged under coverings, it clearly follows that $b(G/H)=b(\widetilde{G}/\bar{H})$.
\end{remark}
%\widetilde{G/H}

We will now prove Theorem~\ref{THM:extension_Nash_Ber}, which states that, if the De Rham decomposition of the universal cover of $G/H$ does not contain any $\mathbb{S}^2$-factor, then the previous conditions (i)-(iv) are also equivalent to $b(G/H)\leq \dim G/H -2$. Thus, Theorem~\ref{THM:extension_Nash_Ber} serves as a gap theorem for normal homogeneous metrics of positive intermediate Ricci curvature. 

\begin{remark}\label{REM:bG_equals_bGxG_DeltaG}
In the proof of Theorem~\ref{THM:extension_Nash_Ber}, we use the computation of $b(G)$ for any semisimple $G$, which is carried out in Section~\ref{SEC:symmetric_spaces} below. Actually we shall compute $b((G\times G)/\Delta G)$, but it holds that $b(G)=b((G\times G)/\Delta G)$, as any normal homogeneous metric (i.e.\ bi-invariant metric) on $G$ is isometric to a normal homogeneous metric (i.e.\ symmetric metric) on $(G\times G)/\Delta G$, via the identification $G\cong (G\times G)/\Delta G$ given by $g\mapsto [g,1]$. See also Proposition~\ref{prop:KxG}~(b) below.
%Actually we shall compute $b((G\times G)/\Delta G)$, but it is straightforward to check that $b(G)=b((G\times G)/\Delta G)$ via the identification $G\cong (G\times G)/\Delta G$ given by $g\mapsto [g,1]$.\marginpar{\tiny Maybe rewrite?} {\color{red}See also Proposition~\ref{prop:KxG} below.} Under this identification, the left $G\times G$-action on $(G\times G)/\Delta G$ corresponds to the two-sided $G\times G$-action $(a,b)\cdot g =agb^{-1}$ on $G$. Consequently, any homogeneous metric on $(G\times G)/\Delta G$ is isometric to a bi-invariant metric on $G$.
\end{remark}

\begin{proof}[Proof of Theorem~\ref{THM:extension_Nash_Ber}]
	Let $G/H$ be a normal homogeneous space for a semisimple Lie group $G$. In view of Remark~\ref{REM:b_coverings}, we can assume that both $G$ and $G/H$ are simply connected, and hence the proof of Theorem~\ref{THM:extension_Nash_Ber} reduces to showing that $b(G/H)=\dim G/H - 1$ if and only if $G/H$ splits off a factor isometric to a round $\mathbb{S}^2$.% (equivalently to $\su{2}/\so{2}$ with a normal metric). 
	
%	To prove sufficiency, suppose $G/H$ is isometric to a Riemannian product $\mathbb{S}^2\times M$ for some $M$. Clearly, $G$ acts transitively and isometrically on each factor and hence are both normal homogeneous. Thus $G/H=(\su{2}/\so{2})\times (G'/H')$ for some $H'<G'<G$.
	To prove sufficiency, suppose $G/H$ is isometric to a Riemannian product $\mathbb{S}^2\times M$ for some $M$. This, by De Rham decomposition theorem for naturally reductive spaces (see~\cite[\S{}X.5]{KN:vol2}), implies that $G/H=G_1/H_1\times G_2/H_2$ for closed subgroups $H_i<G_i<G$, $i\in\{1,2\}$, where $G_1/H_1\cong \mathbb{S}^2$ and $G_2/H_2\cong M$ are endowed with normal homogeneous metrics. Then, by Proposition~\ref{prop:products_homogeneous}:
	$$b(G/H)=\max\{b(\mathbb{S}^2)+\dim M, \dim \mathbb{S}^2+b(M)\}\geq \dim M+1=\dim G/H-1.$$
	It follows that $b(G/H)=\dim G/H-1$ by Nash's result referenced above.
	
	For the necessity, assume $b(G/H)=\dim G/H-1$. By Proposition~\ref{prop:centralizer}, there is $x\in\g{p}$ such that $\dim Z_\g{p}(x)=\dim G/H-1$ and $\dim \Ad(H)x\in\{0,1\}$. We treat these cases separately.
	
	Assume first that there exists $x\in\g{p}$ such that $\dim Z_\g{p}(x)=\dim G/H-1$ and $\Ad(H) x$ has dimension $1$. Then, there exists a normal subgroup $N$ of $H$ such that the induced action of $H/N$ on $\Ad(H) x$ is effective. Since $\Ad(H) x$ is $1$-dimensional, we must have $\dim H/N=1$, as only $1$-dimensional Lie groups can act effectively on a $1$-dimensional manifold. Since $\g{h}$ is a compact (and hence reductive) Lie algebra, it splits as an orthogonal direct sum $\g{h}=\g{so}(2)\oplus\g{n}$, where $\mathfrak{n}$ denotes the Lie algebra of $N$. Let $T$ be a generator of the $\g{so}(2)$-factor. Note that $[T,x]$ spans $T_x(\Ad(H)x)$, and by Proposition~\ref{prop:centralizer} and dimension reasons, $Z_\g{p}(x)=[T,x]^\perp$ is the orthogonal complement of $[T,x]$ in $\g{p}$. Then $Z_\g{g}(x)=\g{n}\oplus[T,x]^\perp$. Indeed, no linear combination of $T$ and $[T,x]$ centralizes $x$, as $\langle [rT+[T,x],x],T\rangle=-\langle [T,x],[T,x]\rangle\neq 0$, $r\in\R$, and $[T,x]\neq 0$. But then $b(G)\geq \dim Z_\g{g}(x)= \dim (\g{n} \oplus Z_\g{p}(x))=\dim G -2$. Theorem~\ref{THM:observations_table_symmetric_spaces}~\eqref{item:Db} and Remark~\ref{rem:factor} imply that $\g{g}$ splits as an orthogonal direct sum $\g{g}=\g{su}_2\oplus \g{g}'$, for some compact semisimple Lie algebra $\g{g}'$, and where the $\g{su}_2$-factor is precisely $\g{su}_2=\R x\oplus Z_\g{g}(x)^\perp=\spann\{T, x, [T,x]\}$. Then $\g{so}_2=\R T\subset\g{su}_2$ and $\g{n}\subset \g{g}\cap \g{su}_2^\perp=\g{g}'$. By assumption $G$ is simply connected and $H$ is connected, thus $G=\su{2}\times G'$ and $H=\so{2}\times N$, where $G'$ is the connected Lie subgroup of $G$ with Lie algebra $\g{g}'$. Therefore $G/H=(\su{2}/\so{2})\times (G'/N)=\mathbb{S}^2\times (G'/N)$.
	
	Now assume that there exists $x\in\g{p}$, $x\neq 0$, such that $\dim Z_\g{p}(x)=\dim G/H-1$ and $\dim \Ad(H) x=0$. Hence $b(G)\geq \dim Z_\g{g}(x)=\dim \g{h}\oplus Z_\g{p}(x)=\dim G-1$, which is impossible by Remark~\ref{rem:factor}. Thus we get a contradiction, and this case is not possible.		
\end{proof}

In the previous proof, an important step was to show that if $b(G/H)=\dim G/H-1$, then $b(G)\geq \dim G-2$. This can be obtained as a consequence of the following result, which, given any homogeneous space $G/H$, provides an upper bound for $b(G/H)$ in terms of $b(G)$, whose value will be computed explicitly for any compact Lie group $G$ in Section~\ref{SEC:symmetric_spaces}. We also state a straightforward lower bound for $b(G/H)$ in terms of the ranks of $G$ and $H$, which generalizes an observation of Berger~\cite[Proposition~6.1]{Be61}.
For a broader generalization of Berger's observation, see \cite[Corollary G]{MoThesis}.

\begin{proposition}\label{prop:inequality}
	Let $G/H$ be a normal homogeneous space. Then
	\[
	\rank G-\rank H\leq b(G/H)\leq \frac{1}{2}\bigl(b(G)+\dim G/H-\dim H\bigr).
	\]
\end{proposition}
\begin{proof}
	Let $x\in\g{p}\setminus \{0\}$. Denote by $H_x$  the isotropy group at $x$ of the restriction of the adjoint representation of $G$ to $H$. In particular, $\Ad(H)x=H/H_x$. Since $Z_\g{p}(x)$ and the Lie algebra of $H_x$ are orthogonal subspaces of $\g{g}$, and both commute with $x$, we have:
	\begin{align*}
		\dim Z_\g{g}(x) &\geq \dim Z_\g{p}(x) + \dim H_x = \dim Z_\g{p}(x) + \dim H - \dim \Ad(H)x 
		\\&= \dim Z_\g{p}(x) + \dim H - (\dim \g{p} - \dim \nu_x (\Ad(H)x)) 
		\\&\geq  2\dim Z_\g{p}(x) + \dim H - \dim G/H,
	\end{align*}
	where in the last inequality we have used Proposition~\ref{prop:centralizer}. Taking the maximum over all $x\in\g{p}\setminus\{0\}$ in both sides of the inequality, we obtain the upper bound for $b(G/H)$ in the statement. The lower bound follows from the fact that any maximal abelian subalgebra $\g{t}$ of $\g{h}$ can be extended to a maximal abelian subalgebra $\g{s}\supset\g{t}$ of $\g{g}$. Indeed, if $x\in\g{s}\cap\g{t}^\perp\subset \g{p}$, $x\neq 0$, then $\g{s}\cap\g{t}^\perp\subset Z_\g{p}(x)$, and hence $b(G/H)\geq \dim Z_\g{p}(x)\geq \dim \g{s}\cap\g{t}^\perp=\rank G-\rank H$.
\end{proof}

%\begin{proposition}\label{prop:inequality}
%	Let $G/H$ be a normal homogeneous space. Then, for each $x\in\g{p}$,
%	\[
%	b(G,x)\geq 2 b(G/H,x)+\dim H-\dim G/H.
%	\]
%	In particular,
%\[
%	 b(G/H)\leq \frac{1}{2}\bigl(b(G)+\dim G/H-\dim H\bigr).
%	\]
%\end{proposition}
%\begin{proof}
%	Let $H_x$ denote the isotropy group at $x$ of the restriction of the adjoint representation of $G$ to $H$. In particular, $\Ad(H)x=H/H_x$. Since $Z_\g{p}(x)$ and the Lie algebra of $H_x$ are orthogonal subspaces of $\g{g}$, and both commute with $x$, we have:
%	\begin{align*}
%	\dim Z_\g{g}(x) &\geq \dim Z_\g{p}(x) + \dim H_x = \dim Z_\g{p}(x) + \dim H - \dim \Ad(H)x 
%	\\&= \dim Z_\g{p}(x) + \dim H - (\dim \g{p} - \dim \nu_x (\Ad(H)x)) 
%	\\&\geq  2\dim Z_\g{p}(x) + \dim H - \dim G/H,
%	\end{align*}
%	where in the last inequality we have used Proposition~\ref{prop:centralizer}.
%\end{proof}

\subsection{Diagonal subgroups and the proof of Theorem~\ref{THM:Wilkings_products}}\label{subsec:thD}

In this subsection we will consider normal metrics on homogeneous spaces of the form $(K\times G)/\Delta K$, where $K<G$ are compact Lie groups and $\Delta K$ is the diagonal embedding of $K$ into the product $K\times G$. Note that $G$ is diffeomorphic to $(K\times G)/\Delta K$ by the identification $g\mapsto [1,g]$. Our study is motivated by the fact, proved in Proposition~\ref{prop:KxG} below, that normal metrics on $(K\times G)/\Delta K$ are ``at least as curved'' as normal metrics (i.e.~bi-invariant metrics) on $G$. Theorem~\ref{THM:Wilkings_products}, which we will prove below, shows that in some cases normal metrics on $(K\times G)/\Delta K$ are indeed ``more curved" than bi-invariant metrics on $G$.

\begin{proposition}\label{prop:KxG}
	Let $G$ be a compact semisimple Lie group, and $K$ a closed subgroup of $G$. Let $\Delta K=\{(k,k):k\in K\}<K\times G$. Then, we have:
	\begin{enumerate}[\rm(a)]
		\item $\max\{b(K),b(G/K)\}\leq b((K\times G)/\Delta K)\leq b(G)$.
		\item If $\rank K=\rank G$, then $b((K\times G)/\Delta K)= b(G)$.
		\item Any homogeneous metric on $(K\times G)/\Delta K$ is isometric to a left-invariant and right $K$-invariant metric on $G$.
	\end{enumerate} 	
\end{proposition}
\begin{proof}
	Let $\langle \cdot,\cdot \rangle$ denote a bi-invariant metric on $G$. We endow $K\times G$ with the bi-invariant product metric $\langle \cdot,\cdot \rangle\vert_\g{k}+\langle \cdot,\cdot \rangle$. We consider the reductive decomposition $\Delta \g{k}\oplus\g{p}$ of $\g{k}\oplus\g{g}$, where $\Delta\g{k}=\{(w,w):w\in\g{k}\}$ and $\g{p}=\{(-y_\g{k},y):y\in\g{g}\}$. Here, $y_\g{k}$ denotes orthogonal projection of $y$ onto $\g{k}$. Let $x=(-y_\g{k},y)$, $x'=(-y_\g{k}',y')\in\g{p}$ be arbitrary. Then $[x,x']=0$ if and only if $[y_\g{k},y_\g{k}']=0=[y,y']$. In other words, $x'\in Z_\g{p}(x)$ if and only if $y_\g{k}'\in Z_\g{k}(y_\g{k})$ and $y'\in Z_\g{g}(y)$. Since the projection $\pi_\g{g}\colon\g{p}\to \g{g}$ onto the second factor of $\g{k}\oplus\g{g}$ is a linear isomorphism, we get that $\dim Z_\g{p}(x)\leq \dim Z_\g{g}(\pi_\g{g}(x))$, for any $x\in\g{p}$, where $\pi_\g{g}(x)\neq 0$ if $x\neq 0$. By Proposition~\ref{prop:centralizer}, this implies that $b((K\times G)/\Delta K)\leq b(G)$, which shows the second inequality in (a).
	
	For the first inequality in (a), let $w\in \g{k}\setminus\{0\}$ be such that $b(K)=\dim Z_\g{k}(w)$. Then $(-w,w)\in\g{p}\setminus\{0\}$, and $Z_\g{p}(-w,w)\supset \{(-w',w'):w'\in Z_\g{k}(w)\}$, from where $b((K\times G)/\Delta K)\geq \dim Z_\g{p}(-w,w)\geq\dim Z_\g{k}(w)=b(K)$. Similarly, let $y\in\g{p}'\defeq\g{g}\cap\g{k}^\perp$, $y\neq 0$, be such that $b(G/K)=\dim Z_{\g{p}'}(y)$. Then $(0,y)\in\g{p}\setminus\{0\}$, $Z_\g{p}(0,y)\supset\{(0,y'):y'\in Z_{\g{p}'}(y)\}$, and hence $b((K\times G)/\Delta K)\geq \dim Z_\g{p}(0,y)\geq\dim Z_{\g{p}'}(y)=b(G/K)$. This completes the prove of Item~(a).
	
	Now assume that $\rank K=\rank G$. Let $\g{t}$ be a maximal abelian subalgebra of $\g{k}$, which by assumption is also a maximal abelian subalgebra of $\g{g}$.  Since centralizers are conjugated under the adjoint representation, and $\g{t}$ intersects any orbit of the adjoint representation of $G$, there is $y\in\g{t}\setminus\{0\}$ such that $\dim Z_\g{g}(y)=\max_{z\in\g{g}\setminus\{0\}}\dim Z_\g{g}(z)=b(G)$. Define $x\defeq(-y,y)=(-y_\g{k},y)\in\g{p}\setminus\{0\}$. Note that if $y'\in Z_\g{g}(y)$, then $0=[y,y']=[y,y'_\g{k}]+[y,y'_{\g{g}\cap\g{k}^\perp}]$, which implies $[y,y'_\g{k}]=0$. 
	Hence, $Z_\g{p}(x)=\{(-y_\g{k}',y'):y'\in Z_\g{g}(y)\}$, which has dimension $\dim Z_\g{g}(y)=b(G)$. Therefore $b((K\times G)/\Delta K)\geq \dim Z_\g{p}(x)=b(G)$, from where (b)~follows.
	
	Finally, note that the natural diffeomorphism $(K\times G)/\Delta K \cong G$ given by $[k,g]\mapsto gk^{-1}$ sends the canonical left $(K\times G)$-action on $(K\times G)/\Delta K$ given by $(k,g)\cdot [a,b]=[ka,gb]$ to the two-sided $(K\times G)$-action on $G$ given by $(k,g)\cdot b=gbk^{-1}$. Thus, any homogeneous metric on $(K\times G)/\Delta K$ is isometric to a metric on $G$ which is left $G$-invariant and right $K$-invariant. This proves (c).
\end{proof}

We will now focus on the particular case considered in Theorem~\ref{THM:Wilkings_products}, namely that of the homogeneous spaces of the form $(K\times K\times K)/\Delta K$, where $K$ is a compact Lie group, and  $\Delta K=\{(k,k,k):k\in K\}$. 
Theorem~\ref{THM:Wilkings_products} asserts that $b(K^3/\Delta K)=2b(K)<b(K^2)$ if $K$ is semisimple, thus showing that both inequalities in Proposition~\ref{prop:KxG}~(a) are strict in this particular~case.
%a normal homogeneous metric on $(K\times K\times K)/\Delta K$ (which is diffeomorphic to $K\times K$) has ``more curvature" than a normal homogeneous metric on $K\times K$, as it shows

%In order to fit with the notation in this section, in the proof below we will use $K$ to refer to the group $G$ in the statement of Theorem~\ref{THM:Wilkings_products}.

\begin{proof}[Proof of Theorem~\ref{THM:Wilkings_products}]
	Let $H\defeq \Delta K<K^3 \;\reflectbox{$\defeq$}\; G$, and consider the homogeneous space $G/H=K^3/\Delta K$. 
	We will prove that $b(G/H)=2b(K)$.
	The corresponding Lie algebras are $\g{g}=\g{k}^3$ and $\g{h}=\{(v,v,v) : v\in\g{k}\}$. On $G=K^3$ we consider the bi-invariant metric given by the product metric of a fixed bi-invariant metric on $K$. Then clearly $\g{p}\defeq \g{h}^\perp =\{ (w,z,-w-z) : w,z\in\g{k}\}$. The Lie bracket splits factor-wise and we can compute the centralizer $\dim Z_{\g{p}}(x)$ of an arbitrary vector $x=(w,z,-w-z)\in\g{p}$ by considering the following subcases:
	\begin{align*}
	w\neq 0,z=0, & & [(w,0,-w),(w',z',-w'-z')]=0 &\Leftrightarrow w',z'\in Z_\g{k}(w), \\
	w=0,z\neq 0, & & [(0,z,-z),(w',z',-w'-z')]=0 &\Leftrightarrow w',z'\in Z_\g{k}(z),  \\
	w,z\neq 0, & & [(w,z,-w-z),(w',z',-w'-z')]=0 &\Rightarrow w'\in Z_\g{k}(w), z'\in Z_\g{k}(z). %, w'+z'\in Z_\g{k}(w+z).
	\end{align*}
	In the first two cases we have $\dim Z_\g{p}(x)=2\dim Z_\g{k}(w)$ and $\dim Z_\g{p}(x)=2\dim Z_\g{k}(z)$, respectively, whereas in the third case we have $\dim Z_\g{p}(x)\leq \dim Z_\g{k}(w)+\dim Z_\g{k}(z)$. Thus, by Proposition~\ref{prop:centralizer} we get
	\[
	b(G/H)=\max_{x\in\g{p}\setminus\{0\}}\dim Z_\g{p}(x)= 2 \max_{w\in\g{k}\setminus\{0\}}\dim Z_\g{k}(w)=2b(K),
	\]
	from where the first claim in Theorem~\ref{THM:Wilkings_products} holds. 
	
	Finally, the claim that $K^2$ admits a metric of $\Ric_{k}>0$ for $k=2b(K)$ which is left-invariant and invariant under right $\Delta K$-diagonal multiplication is a direct consequence of the existence of a (normal) homogeneous metric on $K^3/\Delta K$ with $\Ric_{2b(K)}>0$, along with Proposition~\ref{prop:KxG}~(c). 	
%	Finally, we demonstrate how this shows that $K^2$ admits a metric of $\Ric_{k}>0$ for $k=2b(K)$ which is left-invariant and invariant under right $\Delta K$-diagonal multiplication.
%	Note that the natural diffeomorphism $K^3/\Delta K \cong K^2$ given by $[x,y,z]\mapsto (yx^{-1},zx^{-1})$ sends the canonical left $K^3$-action on $K^3/\Delta K$ to the two-sided $K^3$-action on $K^2$ given by $(x,y,z)\cdot (a,b)=(yax^{-1},zbx^{-1})$. Thus, any homogeneous metric on $K^3/\Delta K$ is isometric to a metric on $K^2$ which is left-invariant and right $\Delta K$-invariant.
\end{proof}

\subsection{Circle bundles over projective spaces}\label{SEC:Witte_spaces}

Here, we discuss properties of a family of homogeneous spaces $M_{k,l}^{p,q}$ that are circle bundles over products of complex projective spaces. We follow the notation in the article by Wang and Ziller \cite{WZ90}.

Consider the product $\sph^{2p+1}\times\sph^{2q+1}$ with $1\leq p\leq q$. The Hopf action on each factor yields a free $\u{1}\times\u{1}$-action on $\sph^{2p+1}\times\sph^{2q+1}$. Let $k,l$ be coprime integers and consider the subgroup $\u{1}^{l,-k}$ from \eqref{EQ:circle_torus} in Subsection~\ref{SS:fat_homogeneous_bundles}. The corresponding quotient is a simply connected homogeneous space \cite[p.~231]{WZ90}, which we denote by
\[
M_{k,l}^{p,q}\defeq \sph^{2p+1}\times_{\u{1}^{l,-k}}\sph^{2q+1}=\frac{\u{p+1}\times\u{q+1}}{\u{1}^{l,-k}\times\u{p}\times\u{q}},
\]
where here $\u{1}^{l,-k}$ is embedded in the $\u{1}\times\u{1}$ given by the $\u{1}$-factors in the block diagonal embedding $(\u{1}\times \u{p})\times (\u{1}\times \u{q})<\u{p+1}\times\u{q+1}$. 
The $7$-dimensional spaces $M_{k,l}^{1,2}$ are commonly known as \emph{Witten spaces}. The space $M_{k,l}^{p,q}$ can also be seen as the total space of a circle bundle over $\CP^p\times\CP^q$ with Euler class $k\alpha_1 + l\alpha_2$, where $\alpha_1,\alpha_2$ denote the generators of $H^2(\CP^p;\ZZ), H^2(\CP^q;\ZZ)$, respectively.

It is a routine calculation to check that, for $kl\neq 0$, normal metrics for the above homogeneous description satisfy
\[
b(M_{k,l}^{p,q})=2q+1, \qquad \text{with } \dim M_{k,l}^{p,q} = 2p+2q+1.
\]
Here we are mostly interested in the case $p=1$, where one can alternatively prove that $b(M_{k,l}^{1,q})=2q+1$ as follows: Proposition~\ref{prop:products_homogeneous} in combination with Remark~\ref{REM:bounds_nested_inclusions} implies
\[
2q+1= b(\CP^1\times\CP^q)\leq b(M_{k,l}^{1,q}) \leq b(\sph^{3}\times\sph^{2q+1})=2q+2.
\]
However, the value $2q+2$ is ruled out by Theorem~\ref{THM:extension_Nash_Ber} since, by the homogeneous description above, it follows that $M_{k,l}^{1,q}$ does not split off isometrically a round $2$-sphere. Now recall that $H^4(M_{k,l}^{1,q};\ZZ)=\ZZ_{l^2}$ if $q>1$ \cite[Proposition~2.1]{WZ90}. By varying $q,l$ we obtain infinitely many homogeneous spaces $G/H$ of distinct homotopy type with $b(G/H)=\dim G/H -2$ in every odd dimension $2q+3\geq 7$, as claimed in Subsection~\ref{SS:normal_homogeneous_symmetric} (discussion below Theorem~\ref{THM:observations_table_symmetric_spaces}).

Observe that if $l,q>1$ then $M_{k,l}^{1,q}$ is not even homotopy equivalent to a product $\sph^2\times X$ \cite[p.~233]{WZ90}. The homotopy sequence associated to $\sph^3\times\sph^{2q+1}\to M_{k,l}^{1,q}$ shows that $\pi_d(M_{k,l}^{1,q})=\pi_d(\sph^2)\oplus\pi_d(\sph^{2q+1})$ for any $d$. If $M_{k,l}^{1,q}$ was homotopy equivalent to $\sph^2\times X$, it would follow that the homotopy groups of $X$ equal those of $\sph^{2q+1}$, and the same conclusion would hold for the homology and cohomology groups by the Hurewicz Theorem and Poincar\'e duality. But then the K\"unneth formula for cohomology would yield that $H^4(\sph^2\times X;\ZZ)$ vanishes, contradicting that $H^4(M_{k,l}^{1,q};\ZZ)=\ZZ_{l^2}$.

The spaces $M_{k,l}^{p,q}$ have very interesting properties and have been used to construct examples exhibiting various kinds of phenomena. For instance, fixing $l,q$ and letting $k$ vary in $\ZZ$ gives only finitely many diffeomorphism types among the family $M_{k,l}^{1,q}$ \cite[p.~217]{WZ90}. Equivalently, there is an infinite subfamily of $M_{k,l}^{1,q}$, all of which are diffeomorphic. Taking one of them as a reference (say $M\defeq M_{k_0,l}^{1,q}$) one can pull-back the normal metric on each of the spaces of the family to $M$ via the corresponding diffeomorphism, thus defining infinitely many elements in the moduli space $\mathcal{M}_{\Ric_{\dim M - 2}>0}(M)$ of metrics of $\Ric_{\dim M - 2}>0$ on $M$. It is obvious that there are natural inclusions
$$
\mathcal{M}_{\Ric_{\dim M - 2} >0}(M) \subset \mathcal{M}_{\Ric >0}(M) \subset \mathcal{M}_{\scal >0}(M)
$$
into the moduli spaces of metrics of $\Ric>0$ and positive scalar curvature, respectively.

Dessai, Klaus and Tuschmann \cite{DKT18} used the Kreck-Stolz invariant to show that, for $q$ even and for many odd values of $l$, the classes of metrics belong to different path components of $\mathcal{M}_{\scal >0}(M)$ and concluded that the same holds for $\mathcal{M}_{\Ric >0}(M)$. Our observation above shows that they actually belong to different path components of $\mathcal{M}_{\Ric_{\dim M - 2} >0}(M)$.

\section{Symmetric Spaces}\label{SEC:symmetric_spaces}
The purpose of this section is to determine, for each symmetric space $G/K$ of compact type, the lowest $k$ for which the symmetric metric on $G/K$ satisfies $\Ric_k>0$. Thus, in Subsection~\ref{subsec:symmetric_prelim} we review the main tools needed and derive the recipe for the calculation of $b(G/K)$. Then, in Subsection~\ref{subsec:symmetric_properties} we discuss further properties of the values $b(G/K)$, from where  Theorem~\ref{THM:observations_table_symmetric_spaces} in Subsection~\ref{SS:normal_homogeneous_symmetric} will follow. We refer the reader to the recent work~\cite{AQZ20} for an alternative discussion of the problem addressed in this section.

\subsection{Root space decomposition and recipe to compute $b(G/K)$}\label{subsec:symmetric_prelim}
Let $M$ be a (not nec\-es\-sarily irreducible) symmetric space of compact type. Let $(G,K)$ be a symmetric~pair representing $M$ and consider the decomposition $\g{g}=\g{k}\oplus\g{p}$, where $\g{p}$ is the orthogonal com\-ple\-ment of $\g{k}$ in $\g{g}$ with respect to the Killing form $B_\g{g}$ of $\g{g}$. As usual, we can identify $\g{p}$ with the tangent space $T_oM$, where $o$ is a fixed point of $K$. We also have $[\g{k},\g{p}]\subset\g{p}$ and~$[\g{p},\g{p}]\subset\g{k}$.

Recall from Proposition~\ref{prop:centralizer} that for each non-zero $x\in \g{p}$, we have
%let us define $b(M,x)$ as the smallest integer $b$ for which the following holds: for every set of orthonormal vectors $y_1,\dots, y_b\in\g{p}$, there is some $i\in\{1,\dots,b\}$ such that $[x,y_i]\neq 0$. Note that $b(M,x)$ is precisely the dimension of the centralizer of $x$ in $\g{p}$, that is,
\[
b(M,x)=\dim Z_\g{p}(x)=\dim\{y\in\g{p}:[x,y]=0\}.
\]
Our purpose is then to determine
\[
b(M)=\max_{x\in\g{p}\setminus\{0\}} b(M,x)
\]
for each symmetric space $M$. We will use a direct approach by means of the classical theory of root systems; for more information on the restricted root decomposition in the compact setting, see~\cite[Chapter~VI]{Loos}. %There are other essentially equivalent approaches, for example by considering the analogous problem in the non-compact dual of $M$ and appealing to the theory of parabolic subalgebras.
We note that the value $b(-)$ is clearly invariant under finite quotients and coverings, i.e., $b(M)=b(\widetilde{M})$, where $\widetilde{M}$ is the universal cover~of~$M$. Moreover, all symmetric pairs $(G,K)$ representing $M$ give rise to the same value $b(M)=b(G/K)$.

%Let $M^*\cong G^*/K^*$ be the dual symmetric space of $M$, which is hence of noncompact type. Let $\g{g}^*=\g{k}\oplus \g{p}^*\subset\g{g}\otimes\C$ be the corresponding Cartan decomposition, where $\g{p}^*=\sqrt{-1}\g{p}$. The bracket in $\g{g}^*$ is defined in the obvious way.  Let $o\in M^*$ be a fixed point of $K^*$, so that $\g{p}^*\cong T_oM^*$. Let $\theta$ be the corresponding Cartan involution, which is determined by $\theta\vert_\g{k}=\id_\g{k}$ and $\theta\vert\g{p^*}=-\id_{\g{p}^*}$. Let $B^*$ denote the Killing form of $\g{g}^*$, which is negative definite on $\g{k}$, positive definite on $\g{p}^*$, and such that $B^*(\g{k},\g{p}^*)=0$. Then $\langle\cdot,\cdot\rangle\defeq -B^*(\theta \cdot,\cdot)$ defines a positive definite inner product on $\g{g}^*$.

%Then, it is immediate that $b(M,x)=\dim Z_{\g{p}^*}(\sqrt{-1}x)$, for all $x\in\g{p}\setminus\{0\}$, and hence $b(M)=\max_{x\in\g{p}^*\setminus\{0\}} \dim Z_{\g{p}^*}(x)$. Moreover, the infinitesimal isotropy representations of $(G,K)$ and $(G^*,K^*)$ are the same (via the identification $\g{p}\cong \sqrt{-1}\g{p}$).

Fix $x\in\g{p}$, $x\neq 0$. Consider a maximal abelian subspace $\g{a}$ of $\g{p}$ containing $x$. Let us put $r\defeq \dim\g{a}=\rank M$. The endomorphisms $\ad(Y)\in\mathrm{End}(\g{g})$, $Y\in\g{a}$, form a commuting family of $B_\g{g}$-skew-adjoint operators, and hence diagonalize simultaneously. Moreover, the self-adjoint operators $\ad(Y)^2$, $Y\in\g{a}$, leave both $\g{k}$ and $\g{p}$ invariant. For any linear form $\lambda\in\g{a}^*$, let us consider 
\begin{align*}
\g{k}_\lambda&=\{X\in\g{k}:\ad(Y)^2X=-\lambda(Y)^2X \text{ for all }Y\in\g{a}\},
\\
\g{p}_\lambda&=\{X\in\g{p}:\ad(Y)^2X=-\lambda(Y)^2X \text{ for all }Y\in\g{a}\}. 
\end{align*}
Then $\g{p}_0=\g{a}$, $\g{p}_{\lambda}=\g{p}_{-\lambda}$, $\g{k}_{\lambda}=\g{k}_{-\lambda}$, and we have the bracket relations $[\g{k}_\lambda,\g{k}_\mu]\subset\g{k}_{\lambda+\mu}+\g{k}_{\lambda-\mu}$, $[\g{k}_\lambda,\g{p}_\mu]\subset\g{p}_{\lambda+\mu}+\g{p}_{\lambda-\mu}$, and $[\g{p}_\lambda,\g{p}_{\mu}]\subset \g{k}_{\lambda+\mu}+\g{k}_{\lambda-\mu}$, for any $\lambda,\mu\in\g{a}^*$. Let us denote by $\Sigma=\{\lambda\in\g{a}^*:\lambda\neq 0,\, \g{p}_\lambda\neq 0\}$ the set of roots of $M\cong G/K$ with respect to $\g{a}\subset\g{p}$. It is well-known that $\Sigma$ constitutes a (possibly non-reduced) root system on the dual space of~$\g{a}$. 
Moreover, we have the $B_\g{g}$-orthogonal decompositions
\[
\g{k}=\g{k}_0\oplus\bigoplus_{\lambda\in\Sigma^+}\g{k}_\lambda\qquad\text{and}\qquad \g{p}=\g{a}\oplus\bigoplus_{\lambda\in\Sigma^+}\g{p}_\lambda,
\]
where $\Sigma^+\subset \Sigma$ is any choice of positive roots (i.e.\ $\Sigma$ is the disjoint union of $\Sigma^+$ and $-\Sigma^+$, and if the sum of two positive roots is a root, then it is also positive).

%Moreover, $[\g{g}_\lambda,\g{g}_\mu]\subset\g{g}_{\lambda+\mu}$ and $\theta\g{g}_\lambda=\g{g}_{-\lambda}$, for all covectors $\lambda$ and $\mu$ of $\g{a}$.

Select $\Lambda=\{\alpha_1,\dots,\alpha_r\}\subset \Sigma$ a set of simple roots for the root system $\Sigma$ such that $x$ belongs to the closed Weyl chamber determined by the relations $\alpha_i\geq 0$, for all $i\in\{1,\dots,r\}$. Then $\Lambda$ determines a set of positive roots $\Sigma^+=\left\{\sum_{i=1}^r a_i \alpha_i\in\Sigma:a_i\in\mathbb{Z}_{\geq 0}\right\}$. Note that if $\lambda=\sum_{i=1}^r a_i \alpha_i\in \Sigma$ is an arbitrary root (where of course the coefficients $a_i\in\mathbb{Z}$ are either all non-negative or all non-positive), then $\lambda(x)=0$ if and only if $\lambda$ is a linear combination of roots in $\Phi=\Phi_x=\{\alpha_i\in\Lambda:\alpha_i(x)=0\}$. Therefore, if we denote by $\Sigma_\Phi^+=\Sigma^+\cap\spann\Phi$ the set of positive roots spanned by $\Phi$, it follows from the definition of the subspaces $\g{p}_\lambda$~that
\begin{equation}\label{eq:Z_p(x)decomp}
Z_{\g{p}}(x)=\g{a}\oplus\bigoplus_{\lambda\in\Sigma^+_\Phi}\g{p}_\lambda.
\end{equation}
This is easily checked to be a Lie triple system (i.e.\ $[[Z_{\g{p}}(x),Z_{\g{p}}(x)],Z_{\g{p}}(x)]\subset Z_{\g{p}}(x)$), and hence $Z_{\g{p}}(x)\cong T_o F_\Phi$ is the tangent space of a totally geodesic submanifold $F_\Phi$ containing $o\in M$. As a totally geodesic submanifold of a symmetric space of compact type, $F_\Phi$ is a symmetric space of non-negative curvature. Indeed, in this case, $F_\Phi$ is (up to finite quotients) of the form $B_\Phi \times \mathbb{T}^{r-|\Phi|}$, where $B_\Phi$ is the compact semisimple factor with associated Lie triple system $(\oplus_{\lambda\in\Phi}Y_\lambda)\oplus(\oplus_{\lambda\in\Sigma^+_\Phi}\g{p}_\lambda)$, where $Y_\lambda\in\g{a}$ is determined by $B_\g{g}(Y_\lambda, Y)=\lambda(Y)$ for all $Y\in\g{a}$, and $\mathbb{T}^{r-|\Phi|}$ is the Euclidean factor with Lie triple system $\cap_{\lambda\in\Phi}\ker\lambda\subset\g{a}$. 
 
%where $B_\Phi$ is the so-called boundary component of $M^*$ correspoding to the subset $\Phi\subset\Lambda$ (in the context of the maximal Satake compactification of $M^*$, see~\cite[\S{}I.1]{BJ}).
%As a totally geodesic submanifold of a symmetric space (of noncompact type), $F_\Phi$ is intrinsically a symmetric space with noncompact semisimple factor $B_\Phi$ and Euclidean factor $\R^{r-|\Phi|}$.

Since our aim is to determine $b(M)=\max_{x\in\g{p}\setminus\{0\}} \dim Z_{\g{p}}(x)$, in view of \eqref{eq:Z_p(x)decomp} we can restrict our attention to maximal proper subsets $\Phi$ of $\Lambda$. Indeed, if $x$ is such that $\Phi_x=\Lambda$, this implies that $x=0$, since $\Lambda$ is a basis of the dual space of $\g{a}$; and if $|\Phi_x|\leq r-2$, then we can consider $\Phi'\subsetneq \Lambda$ properly containing $\Phi_x$, and take a non-zero $x'\in\cap_{\alpha_i\in\Phi'}\ker\alpha_i$, which would then satisfy $Z_{\g{p}}(x)\subset Z_{\g{p}}(x')$ by \eqref{eq:Z_p(x)decomp}.

Thus, let us put $\Phi=\Phi^k=\Lambda\setminus\{\alpha_k\}$, for each $k\in\{1,\dots,r\}$. Then we can choose $x=Y^k$, where $Y^k$ is the vector in $\g{a}$ such that $\alpha_i(Y^k)=\delta_{ik}$, for each $i\in\{1,\dots,r\}$. We have to determine $\dim Z_{\g{p}}(x)=\dim B_{\Phi^k} \times \mathbb{T}^1$ for each $k$, and then $b(M)$ will be the maximum of all these values, where $k$ runs through $\{1,\dots,r\}$. This is an easy combinatorial problem that can be done with the help of the Dynkin diagram of the root system $\Sigma$ of $M$ along with the multiplicities of the roots (i.e.\ the dimensions of each $\g{p}_\lambda$). The approach is the following. In the Dynkin diagram of $M$ delete the node corresponding to $\alpha_k$  (and any edge connected to it). This gives rise to a new Dynkin diagram which, together with the multiplicities attached to each one of the remaining nodes, must then correspond in a unique way to some symmetric space of compact type (up to finite quotients and coverings). Such symmetric space is precisely $B_{\Phi^k}$. Finally, we observe that, for each symmetric space $M$ and each $k\in\{1,\dots,r\}$, $\dim B_{\Phi^k}$ takes the value of a quadratic polynomial in $k$ with positive leading coefficient. Thus its maximum is achieved at $k=1$ or $k=r$, from where we directly obtain $b(M)$ by just adding $1$ to such maximum. All this study can be done case by case for each irreducible symmetric space; see Table~\ref{table:b_symmetric} at the end of the article.

We note that Table~\ref{table:b_symmetric} contains some redundancies stemming from the isomorphisms existing between some low dimensional symmetric spaces. These isomorphisms can be consulted in~\cite[pp.~519-520]{Helgason}, but can also be derived from Table~\ref{table:b_symmetric}, as the Dynkin diagram and the multiplicities of the simple roots determine the simply connected compact symmetric~space.

%\begin{remark}\label{rem:parabolic}
%	There is a nice interpretation of the submanifolds $F_\Phi$ and $B_\Phi$ in terms of the theory of parabolic subgroups of real semisimple Lie groups. Indeed, the Lie triple system~\eqref{eq:Z_p(x)decomp} corresponds in a natural way to a Lie triple system of the non-compact dual $M^*$ of $M$, thus giving rise also to a totally geodesic submanifold $F_\Phi^*=B_\Phi^*\times\R^{r-|\Phi|}$ of $M^*$. This $B_\Phi^*$ is the so-called boundary component of $M^*$ associated with the subset $\Phi\subset\Lambda$ of simple roots (in the context of the maximal Satake compactification of $M^*$, see~\cite[\S{}I.1]{BJ}). Moreover, if $G^*$ is the connected component of the identity of the isometry group of $M^*$ (which is a non-compact real semisimple Lie group), and $L_\Phi$ and $M_\Phi$ denote the reductive and semisimple parts of the Chevalley and Langlands decompositions (respectively) of the parabolic subgroup of $G^*$ associated to the subset $\Phi$ of simple restricted roots, then $F_\Phi^*$ and $B_\Phi$ are precisely the orbits through the base point of $M^*$ of the actions of $L_\Phi$ and $M_\Phi$, respectively (see for example~\cite[\S2]{BO:jdg}).
%\end{remark}

\subsection{Discussion of particular cases and proof of Theorem~\ref{THM:observations_table_symmetric_spaces}}\label{subsec:symmetric_properties}
We now discuss the values $b(M)$ in various relevant cases.  Theorem~\ref{THM:observations_table_symmetric_spaces} will follow from this discussion.

First, observe that the first claim in Theorem~\ref{THM:observations_table_symmetric_spaces} just refers to Table~\ref{table:b_symmetric} for the values of~$b(M)$ of irreducible symmetric spaces $M$, whereas the second claim follows from a direct application of Proposition~\ref{prop:products_homogeneous} to the De Rham decomposition $\widetilde{M}=M_1\times \dots \times M_s$ of the universal cover of $M$ into irreducible factors.

Let us prove Item~\eqref{item:Da} in Theorem~\ref{THM:observations_table_symmetric_spaces}. We start by noting that $b(M)=1$ if and only if $M$ is of rank one, but if $M$ has higher rank, then $b(M)\geq 2\rank M-1\geq 3$. Indeed, for any maximal proper $\Phi\subset\Lambda$ we have $\dim \oplus_{\lambda\in \Sigma^+_\Phi}\g{p}_\lambda\geq |\Phi|= r-1$; then from \eqref{eq:Z_p(x)decomp} it follows that 
\begin{equation}\label{eq:lower_bound}
b(M)\geq \dim\g{a} +\dim \oplus_{\lambda\in \Sigma^+_\Phi}\g{p}_\lambda\geq r+(r-1)=2r-1.
\end{equation}
%Now assume $b(M)=\dim M-1$. Then there exists $x\in\g{a}$ such that $\dim Z_{\g{p}}(x)=\dim \g{p}-1$. By \eqref{eq:Z_p(x)decomp} we must have that there is exactly one positive root $\alpha\in \Sigma^+$ such that $\g{p}_\alpha$ is not contained in (and is orthogonal to) $Z_\g{p}(x)$, and we have $\dim \g{p}_\alpha=1$ for such root. Thus, $\alpha\in\Sigma^+\setminus\Sigma^+_\Phi$. By the properties of root systems, the root system $\Sigma$ is then reducible with a rank one factor (the span of $\alpha$). Thus, the universal cover of $M$ has a rank one factor, which must necessarily be isometric to $\mathbb{S}^2$, since $\dim \g{p}_\alpha=1$. Conversely, any simply connected symmetric space with a factor isometric to $\mathbb{S}^2$ 
%If $M$ is reducible, let $\widetilde{M}=M_1\times \dots \times M_s$ be the decomposition into irreducible factors of the universal cover of $M$. Decompose $\Lambda=\Lambda_1\sqcup\dots\sqcup\Lambda_s$ accordingly, and suppose that $\alpha_k$ belongs to $\Lambda_{j(k)}$. Then $\Phi^k=\Lambda\setminus\{\alpha_k\}$ contains the simple roots corresponding to all irreducible factors of $\widetilde{M}$ except exactly one, namely $M_{j(k)}$. Hence $B_{\Phi^k}$ is (again, up to quotient) the product of all irreducible factors of $M$ except $M_{j(k)}$, times the totally geodesic submanifold $B_{\Lambda_{j(k)}\setminus\{\alpha_k\}}$ of $M_{j(k)}$. Thus, $b(M)$ is the maximum of the sum of $b(M_j)$ and the dimensions of the remaining irreducible factors $M_i$, $i\neq j$. 
Assume $b(M)=2\rank M-1$. Then, in view of~\eqref{eq:lower_bound}, for any maximal proper $\Phi\subset\Lambda$ we must have $\lvert\Sigma^+_\Phi\rvert=r-1$ (and hence $\Sigma_\Phi^+=\Phi$) and $\dim\g{p}_\lambda=1$ for all $\lambda\in\Sigma_\Phi^+$. This means that, for any maximal proper $\Phi\subset\Lambda$, $B_\Phi$ is a symmetric space of type $(\mathsf{A}_1)^{r-1}$ (equivalently, when one removes a node from the Dynkin diagram of $M$, one obtains a totally disconnected graph with no double nodes) and with all multiplicities equal to $1$. If $r=1$, this is trivially satisfied. If $M$ is irreducible, it easily follows from the Dynkin diagrams and multiplicities in Table~\ref{table:b_symmetric} that $M$ must be isometric to $\su{3}/\so{3}$, $\grp{2}{5}{\R}$ or $\mathsf{G}_2/\so{4}$. If $M$ is reducible, all factors have to be of type $\mathsf{A}_1$ with multiplicity $1$, that is, $M$ is covered by a product of $2$-spheres. All in all, this proves Theorem~\ref{THM:observations_table_symmetric_spaces}~\eqref{item:Da}.

Now assume $b(M)=\dim M-1$. Then, there exists $x\in\g{a}$ such that $\dim Z_{\g{p}}(x)=\dim \g{p}-1$. By \eqref{eq:Z_p(x)decomp} there is exactly one positive root $\alpha\in \Sigma^+$ such that $\g{p}_\alpha$ is not contained in (and is orthogonal to) $Z_\g{p}(x)$, and we have $\dim \g{p}_\alpha=1$ for such root. Thus, $\{\alpha\}=\Sigma^+\setminus\Sigma^+_\Phi$. By the properties of root systems, $\Sigma$ is then reducible with a rank one factor (the span of $\alpha$). Thus, the universal cover of $M$ has a rank one factor, which must necessarily be isometric to $\mathbb{S}^2$, since $\dim \g{p}_\alpha=1$. The converse is direct from Proposition~\ref{prop:products_homogeneous}.

Let us consider the case $b(M)=\dim M-2$. Then $\dim Z_{\g{p}}(x)=\dim \g{p}-2$ for some $x\in\g{a}$, and by \eqref{eq:Z_p(x)decomp} we deduce that either there is exactly one $\alpha\in\Sigma^+\setminus\Sigma_\Phi^+$, and it satisfies $\dim\g{p}_\alpha=2$, or there are two different roots $\alpha,\beta$ such that $\{\alpha,\beta\} =\Sigma^+\setminus\Sigma_\Phi^+$ and both with multiplicity~$1$. In the first case, similarly as above, the universal cover of $M$ splits off a factor isometric to $\su{2}=\mathbb{S}^3$. Let us assume we are in the second case. If the roots in $\Sigma^+\setminus\spann\{\alpha,\beta\}$ are perpendicular to $\spann\{\alpha,\beta\}$, then $\Sigma$ is reducible with a rank two factor of type $\mathsf{A}_2$ (the span of $\alpha$ and $\beta$); since the multiplicities of the roots of this factor are all $1$, $M$ locally splits off a factor isometric to $\su{3}/\so{3}$. On the contrary, if there is a root in $\Sigma^+\setminus\spann\{\alpha,\beta\}$ not perpendicular to $\spann\{\alpha,\beta\}$, we can assume without restriction of generality that there exists $\lambda\in \Sigma\setminus\spann\{\alpha,\beta\}$ not perpendicular to $\alpha$, and also $\langle \lambda,\alpha\rangle >0$. Then, again by the properties of root systems, we have that $\lambda-\alpha\in\Sigma$, but $\lambda-\alpha\notin\Sigma_\Phi$ (since $\lambda\in\Sigma_\Phi$ and $\alpha\notin\Sigma_\Phi$). This would imply that $\lambda-\alpha\in\{\alpha,\beta,-\alpha,-\beta\}$, which contradicts $\lambda\notin\spann\{\alpha,\beta\}$. The discussion in this paragraph and the previous one implies Theorem~\ref{THM:observations_table_symmetric_spaces}~\eqref{item:Db}.

\begin{remark}\label{rem:factor}
	In the particular case where $M$ is a compact semisimple Lie group $G$, the discussion in the previous two paragraphs implies that $b(G)=\dim G-1$ is impossible, whereas $b(G)=\dim G-2$ if and only if $\g{g}$ is a Lie algebra direct sum of the form $\g{su}_2\oplus\g{g}'$, for some compact semisimple  $\g{g}'$. Even more, if for some $x\in\g{g}$ we have $\dim Z_\g{g}(x)=\dim\g{g}-2$, then $(\R x\oplus Z_\g{g}(x)^\perp)\oplus (Z_\g{g}(x)\cap x^\perp)$ yields precisely the splitting $\g{su}_2\oplus\g{g}'$.
\end{remark}
 
In order to prove Items~\eqref{item:Dc} and~\eqref{item:Dd} of Theorem~\ref{THM:observations_table_symmetric_spaces}, we first show that if $b(M)\leq (\dim M)/2$, then $M$ is irreducible. Indeed, by Proposition~\ref{prop:products_homogeneous}, $b(M)=\max\{b(M_j)-\dim M_j\}+\dim M$, where $\widetilde{M}=M_1\times \dots\times M_s$. Then $b(M)\leq (\dim M)/2$ is equivalent to $\min\{\dim M_j-b(M_j)\}\geq (\dim M)/2$, from where we get $\dim M_j\geq (\dim M)/2$ for all $j\in\{1,\dots, s\}$. If $M$ is reducible, this necessarily implies $\widetilde{M}=M_1\times M_2$ with $\dim M_1=\dim M_2$. But then, assuming without restriction of generality that $b(M_1)\leq b(M_2)$, we would have $\dim M_1=(\dim M)/2\geq b(M)=b(M_2)+\dim M_1>\dim M_1$, which yields a contradiction.
The rest of claims in Theorem~\ref{THM:observations_table_symmetric_spaces} \eqref{item:Dc}-\eqref{item:Dd} follow from elementary case-by-case calculations making use of the values of $\dim M$ and $b(M)$ stated in Table~\ref{table:b_symmetric} at the end of the article.

For convenience of the reader, we list the simply connected irreducible symmetric spaces~$M$ of rank at least $2$ with $b(M)\leq 6$ in Table~\ref{table:low_b_symmetric} below, as referenced in Theorem~\ref{THM:observations_table_symmetric_spaces}~\eqref{item:De}.

%\begin{corollary}
%	The irreducible symmetric spaces $M$ with $3\leq b(M)\leq 6$ are listed in Table~\ref{table:low_b_symmetric}.
%\end{corollary}
\begin{table}[h!]
	\caption{Irreducible symmetric spaces with $3\leq b(M)\leq 6$.}
	\label{table:low_b_symmetric}
	\begin{tabular}{c@{\hspace{10ex}}c}
	\begin{tabular}{ccc}\hline
		$M$ & $\dim M$ & $b(M)$
		\\ \hline
		$\su{3}/\so{3}$ & $5$ & \multirow{3}{*}{$3$}
		\\
		$\grp{2}{5}{\R}$ & $6$ &
		\\
		$\mathsf{G}_2/\so{4}$ & $8$
		\\ \hline
		$\su{3}$ & $8$ & \multirow{4}{*}{$4$}
		\\
		$\grp{2}{6}{\R}$ & $8$ &
		\\
		$\sp{2}$ & $10$ & 
		\\
		$\mathsf{G}_2$ & $14$ &
		\\ \hline
	\end{tabular}
&
	\begin{tabular}{ccc}\hline
	$M$ & $\dim M$ & $b(M)$
	\\ \hline
	$\grp{2}{7}{\R}$ & $10$ & \multirow{2}{*}{$5$}
	\\
	$\gr{2}{5}{\C}$ & $12$ &
	\\ \hline
	$\su{4}/\so{4}$ & $9$ & \multirow{4}{*}{$6$}
	\\
	$\grp{2}{8}{\R}$ & $12$ & 
	\\
	$\su{6}/\sp{3}$ & $14$ &
	\\
	$\gr{2}{4}{\H}$ & $16$ &
	\\ \hline
	\end{tabular}
	\end{tabular}
\end{table}

In the case of rank $2$ symmetric spaces, Proposition~\ref{prop:centralizer} implies that $b(M)$ is the highest codimension in $T_{o}M$ of an orbit of the isotropy representation of $M=G/K$, restricted to the unit sphere of $T_{o}M$. Such restrictions of isotropy representations of rank $2$ symmetric spaces give rise to all cohomogeneity one actions on round spheres (up to orbit equivalence), or equivalently, to all homogeneous isoparametric families of hypersurfaces on round spheres~\cite[\S2.9.6]{BCO16}. Any such action has exactly two singular orbits, and the remaining ones are hypersurfaces of the sphere with $g\in\{1,2,3,4,6\}$ constant principal curvatures and multiplicities $m_i$ satisfying $m_i= m_{i+2}$ (indices modulo $g$). (The integer $g$ agrees with the number $|\Sigma^+|$ of positive roots of $M$, and the $m_i$ are the multiplicities of such roots.) The codimension (in the sphere) of a singular orbit coincides with $m_i+1$, for some $i\in\{1,\dots, g\}$. Hence, $b(M)=2+\max_{i\in\{1,\dots, g\}} m_i$. In Table~\ref{table:b_symmetric_rank2}, all simply connected compact semisimple symmetric spaces $M$ of rank $2$ are shown, together with their dimensions, the number $g$ of principal curvatures and the multiplicities of the corresponding homogeneous hypersurfaces, from where one directly obtains $b(M)$. 

\begin{table}[h!]
	\caption{Symmetric spaces of rank two.}\label{table:b_symmetric_rank2}
	\begin{tabular}{ccccc}\hline
		$g$ & Multiplicities & $M$ & $\dim M$ & $b(M)$
		\\ \hline
		%	$1$ & $l-2$ & $\mathbb{S}^1\times \mathbb{S}^{l-1}$ & $l$ & $l$
		%	\\
		$2$ & $(k, l-k-2)$ & $\mathbb{S}^{k+1}\times \mathbb{S}^{l-k-1}$ & $l$ & $\max\{k+2, l-k\}$
		\\
		$3$ & $1$ & $\su{3}/\so{3}$ & $5$ & $3$
		\\
		$3$ & $2$ & $\su{3}$ & $8$ & $4$
		\\
		$3$ & $4$ & $\su{6}/\sp{3}$ & $14$ & $6$
		\\
		$3$ & $8$ & $\mathsf{E}_6/\mathsf{F}_4$ & $26$ & $10$
		\\
		$4$ & $(2,2)$ & $\sp{2}$ & $10$ & $4$
		\\
		$4$ & $(4,5)$ & $\so{10}/\u{5}$ & $20$ & $7$
		\\
		$4$ & $(1,k-2)$  & $\grp{2}{k+2}{\R}$, $k\geq 3$ & $2k$ & $k$
		\\
		$4$ & $(2,2k-3)$ & $\gr{2}{k+2}{\C}$, $k\geq 2$     & $4k$ & $\max\{4,2k-1\}$
		\\
		$4$ & $(4,4k-5)$ & $\gr{2}{k+2}{\H}$, $k\geq 2$   & $8k$ & $\max\{6,4k-3\}$
		\\
		$4$ & $(9,6)$ & $\mathsf{E}_6/\spin{10} \u{1}$ & $32$ & $11$
		\\
		$6$ & $(1,1)$ & $\mathsf{G}_2/\so{4}$ & $8$ & $3$
		\\
		$6$ & $(2,2)$ & $\mathsf{G}_2$ & $14$ & $4$
		\\
		\hline
	\end{tabular}
\end{table}

%% file: fat_bundles.tex
\section{Fat Homogeneous bundles: Metrics and topology}\label{SEC:fat_bundles}
	
The aim of this section is to investigate the $k^\textrm{th}$-intermediate Ricci curvature of the total spaces of certain fat bundles, as well as the topology of some of them. 
	In Subsection~\ref{SS:fat_bundles} we recall the notion of fat homogeneous bundle and prove Theorem~\ref{THM:Wallach_generalization}. 
	 In Subsection~\ref{SS:fat_bundles_examples} we discuss various families of fat homogeneous bundles and obtain  Theorem~\ref{THM:homogeneous_spaces_fat} as a consequence, which in particular implies Theorem~\ref{THM:homogeneous_simple_version}. In Subsection~\ref{SS:fat_bundles_topology} we compute the topology of the generalized Aloff-Wallach spaces.
	
%Throughout this section $G$ will denote a compact, connected Lie group.
	
	\subsection{Metrics on fat homogeneous bundles}\label{SS:fat_bundles}
		
				The notion of \emph{fatness} was introduced by Weinstein to study Riemannian submersions with totally geodesic fibers and positive vertizontal curvatures; see the discussion in \cite[Section 2.8]{GW09} or \cite{Zi99,FZ11}. In this work we are only interested in the case where the submersion is given by a \emph{homogeneous bundle}. Recall that a homogeneous bundle is the bundle $K/H\to G/H\to G/K$ resulting from a nested triple of compact Lie groups $H<K<G$. Let $\mathfrak{h}\subset\mathfrak{k}\subset\mathfrak{g}$ be the respective Lie algebras. We fix a bi-invariant metric $\langle\cdot,\cdot\rangle$ on $G$ and use it to define:
		\[
			\mathfrak{m}\defeq\mathfrak{h}^\perp \cap\mathfrak{k}, \qquad \mathfrak{p}\defeq \mathfrak{k}^\perp.
		\]
		Recall that there are canonical identifications
		\[
			T_{eH} K/H\simeq \mathfrak{m},\qquad T_{eK} G/K\simeq\mathfrak{p}  ,\qquad T_{eH} G/H\simeq\mathfrak{h}^\perp=\mathfrak{m}\oplus\mathfrak{p}.
		\]
With this notation, there is the following characterization of fat homogeneous bundles which we take as a definition (see \cite[Proposition~2.8.3]{GW09}).

	%	We will now setup and prove Theorem~\ref{THM:Wallach_generalization}.
	%	As in the Introduction, let $H<K<G$ be a nested triple of compact Lie groups and denote by $\mathfrak{h}\subset\mathfrak{k}\subset\mathfrak{g}$ their Lie algebras. 
	
\begin{definition}\label{DEF:fat}
A homogeneous bundle is fat if $[x,y]\neq 0$ for all non-zero $x\in\mathfrak{m}$ and~$y\in\mathfrak{p}$. 
\end{definition}
		
We will now describe the metrics that will be used in the proof of Theorem~\ref{THM:Wallach_generalization}. For any $t>0$, consider the bi-invariant product metric $(K\times G,t\langle\cdot,\cdot\rangle|_\mathfrak{k} + \langle\cdot,\cdot\rangle)$. The quotient with respect to the action by right multiplication $k_1(k_2,g)=(k_2k_1^{-1},gk_1^{-1})$, $g\in G$, $k_1,k_2\in K$, of the diagonal subgroup $\Delta K< K\times G$ is a normal homogeneous metric on $(K\times G)/\Delta K$, denoted by $\langle\cdot,\cdot\rangle_t$. Note that Proposition~\ref{prop:KxG} applies in this case. Hence, the natural diffeomorphism $(K\times G)/\Delta K \cong G$ given by $[k,g]\mapsto gk^{-1}$ sends the action by left multiplication $(k_1,g_1)[k_2,g_2]=[k_1k_2,g_1g_2]$ to the $(K\times G)$-action on $G$ given by $(k_1,g_1)g_2=g_1g_2k_1^{-1}$. In particular, the push-forward of the left $(K\times G)$-invariant metric $\langle\cdot,\cdot\rangle_t$ on $(K\times G)/\Delta K$ results in a left-invariant and right $K$-invariant metric on $G$, which we also denote by $\langle\cdot,\cdot\rangle_t$. This discussion can be summarized in:
%		\begin{equation}\label{eq:gbit}		
$$
			(K\times G,t\langle\cdot,\cdot\rangle|_\mathfrak{k} + \langle\cdot,\cdot\rangle) \to ((K\times G)/\Delta K,\langle\cdot,\cdot\rangle_t) \cong (G,\langle\cdot,\cdot\rangle_t),
$$
where the first arrow is a Riemannian submersion. 

%		\end{equation}	
%		For the interested reader, we remark that the construction above can be rephrased by saying that $(G,\langle\cdot,\cdot\rangle_t)$ is the Cheeger deformation of $(G,\langle\cdot,\cdot\rangle)$ with respect to the action of $K<G$ by right multiplication and to the metric $\langle\cdot,\cdot\rangle|_\mathfrak{k}$ on $K$, see \cite[Section~4]{Zi07}.
		The right action by $H$ on $(G,\langle\cdot,\cdot\rangle_t)$ is by isometries since $H<K$, and hence $\langle\cdot,\cdot\rangle_t$ induces a metric $q_t$ on $G/H$ such that the quotient map
		\begin{equation}\label{eq:qt}
			(G,\langle\cdot,\cdot\rangle_t) \to (G/H, q_t)
		\end{equation}
		is a Riemannian submersion. 
		Since $\langle\cdot,\cdot\rangle_t$ is left-invariant and the left $G$-action on $G$ commutes with the right $H$-action, it follows that $q_t$ is $G$-invariant, i.e. $q_t$ is homogeneous on $G/H$. Observe that $\langle\cdot,\cdot\rangle_t$ and $q_t$ are of $\sec\geq 0$ by O'Neill's formula. We will need the following result showing that, in some situations, $\langle\cdot,\cdot\rangle_t$ has fewer zero-curvature planes than the bi-invariant metric $\langle\cdot,\cdot\rangle$; see \cite[Satz~231]{Es84} or \cite[Lemma 4.2]{Zi07}.
		
		\begin{theorem}[Eschenburg]\label{THM:Eschenburg}
			Assume $(G,K)$ is a compact symmetric pair (i.e. $[\mathfrak{p},\mathfrak{p}]\subset\mathfrak{k}$). 
			Then, for any $t>0$, the vectors $x,y\in\mathfrak{g}$ span a zero-curvature plane with respect to $\langle\cdot,\cdot\rangle_t$ if and only if the following is satisfied:
			\[
				[x,y]=[x_\mathfrak{k},y_\mathfrak{k}]=[x_\mathfrak{p},y_\mathfrak{p}]=0.
			\]
		\end{theorem}
		
		We are ready to prove Theorem~\ref{THM:Wallach_generalization}, which we restate here in a more precise manner:
		
		\begin{theorem}\label{THM:Wallach_generalization_refined}
			Let $H < K < G$ be a fat homogeneous bundle with $G/K$ a symmetric space. 
			Then, for any $t>0$, the homogeneous metric $q_t$ defined on $G/H$ as in \eqref{eq:qt} has $\Ric_k>0$ for $k=b(G/K)$.
		\end{theorem}
		\begin{proof}
			We argue by contradiction. 
			Suppose that there are orthonormal vectors 
			$$\{x,y^1,\dots,y^k\}\subset\mathfrak{m}\oplus\mathfrak{p}\simeq T_{eH}(G/H)$$ satisfying $\sum \sec_{q_t}(x,y^i)\leq 0$. Since $q_t$ is of $\sec\geq 0$ it follows that $\sec_{q_t}(x,y^i)=0$ for all $i$. 
			
			Observe that, being $(G,\langle\cdot,\cdot\rangle_t)$ of $\sec\geq 0$ and $(G,\langle\cdot,\cdot\rangle_t)\to (G/H,q_t)$ a Riemannian submersion, it follows that the planes spanned by $x,y^i$ in $\mathfrak{m}\oplus\mathfrak{p}\subset\mathfrak{g}$ are also flat with respect to $\langle\cdot,\cdot\rangle_t$. Since $G/K$ is symmetric, by Eschenburg's Theorem~\ref{THM:Eschenburg} we have that for all $i$,
			$$
			[x,y^i]=[x_\mathfrak{k},y^i_\mathfrak{k}]=[x_\mathfrak{p},y^i_\mathfrak{p}]=0.
			$$
			The identities $[x_\mathfrak{p},y^i_\mathfrak{p}]=0$ together with the assumption $b(G/K)=k$ imply that either:
			\begin{itemize}
			\item there is $i_0$ such that $x_\mathfrak{p},y^{i_0}_\mathfrak{p}$ are linearly dependent, or 
			\item the vectors ${y_\mathfrak{p}^1,...,y_\mathfrak{p}^k}$ are linearly dependent.
			\end{itemize}
			In the second case we can rearrange the orthonormal vectors ${y^1,...,y^k}$ so that $y_\mathfrak{p}^1=0$. Thus, in either case we can assume that there is $i_0$ such that $x_\mathfrak{p},y^{i_0}_\mathfrak{p}$ are linearly dependent.
			
			Now we look at the identity $[x_\mathfrak{k},y^{i_0}_\mathfrak{k}]=0$. Recall that $x,y^{i_0}\in\mathfrak{m}\oplus\mathfrak{p}$ and $\mathfrak{m}=\mathfrak{h}^\perp \cap\mathfrak{k}$, hence $x_\mathfrak{k}=x_\mathfrak{m}$  and $y^{i_0}_\mathfrak{k}=y^{i_0}_\mathfrak{m}$, and in particular $[x_\mathfrak{m},y^{i_0}_\mathfrak{m}]=0$. Since the bundle is fat, it follows from \cite[Lemma 6]{Be75} that $x_\mathfrak{m},y^{i_0}_\mathfrak{m}$ must be linearly dependent as well\footnote{Note that Bérard-Bergery uses a different notation with respect to ours: we interchange $H$ with $K$ and $\mathfrak{m}$ with~$\mathfrak{p}$.}. Hence we can rearrange the vectors $x,y^{i_0}$ so that either $x\in\mathfrak{p}$ and $y^{i_0}\in\mathfrak{m}$ or $x\in\mathfrak{m}$ and $y^{i_0}\in\mathfrak{p}$. In either case, the last condition $[x,y^{i_0}]=0$ together with Definition~\ref{DEF:fat} of fatness imply that $x,y^{i_0}$ are linearly dependent, which is a contradiction.
		\end{proof}

\subsection{Examples}\label{SS:fat_bundles_examples}

The classification of fat homogeneous bundles was achieved in \cite{Be75}, it was revisited in \cite[Section~3]{Zi99} and some cases were used in \cite[Section 6]{FZ11}. In this article we consider various such bundles, which can be divided in three different sets.

The first set consists of those bundles induced by the following nested inclusions, where $p,q$ are coprime integers with $pq>0$ and $n\geq 2$:
\[
\renewcommand{\arraycolsep}{0.5ex}
\begin{array}{clclcl}
\sss (\u{1}^{p,q}\times \u{n-1}) & <& \sss (\u{2}\times \u{n-1})&<& \su{n+1},  \\
\sp{1}\times \sp{1}\times\sp{n-1} & <& \sp{2}\times\sp{n-1} &<& \sp{n+1},\\
\sss (\u{2}\times \sp{2}) & <& \sss (\u{2}\times\u{4}) &<& \su{6},\\
\so{2}\times \spin{7} & <& \so{2}\times\so{8} &<& \so{10}.\\
\end{array}
\]
These inclusions correspond to (D.29), (C.23), (C.21) and (C.24) from \cite[Section~7]{Be75} and to (E.6), (E.5.c), (E.5.a) and (E.5.b) from \cite[Section~3]{Zi99}, respectively. The two first rows are also discussed in \cite[Section 6]{FZ11}. The inclusions are given by standard diagonal block embeddings except for the first one (explained in Subsection~\ref{SS:fat_homogeneous_bundles}), $\sp{2}<\su{4}$ (equivalent to the standard embedding $\spin{5}<\spin{6}$ via the corresponding isomorphisms) and $\spin{7} <\so{8}$ (induced by the irreducible representation of $\spin{7}$ on $\RR^8$). The base spaces of the corresponding fat bundles are $\gr{2}{n+1}{\C}$, $\gr{2}{n+1}{\H}$, $\gr{2}{6}{\C}$ and $\grp{2}{10}{\R}$, respectively. The total space $\sp{n+1}/(\sp{1}^2\times\sp{n-1})$ corresponding to the second row equals the projectivized tangent bundle of $\HP^n$, denoted by $\mathbb{P}_\HH T\HP^n$, and the low-dimensional case $\mathbb{P}_\HH T\HP^2$ is by definition the Wallach flag manifold $W^{12}$.

The second set comes from the following source. Recall that a compact irreducible symmetric space $G/K$ is Hermitian if $K$ locally splits off a circle factor $K=K'\u{1}$ (see \cite[p.~518]{Helgason} for their classification). The nested inclusion $K' <  K'\u{1} < G$ induces a fat homogeneous (circle) bundle; see \cite[Section~4]{Be75} or \cite[p.~17]{Zi99}. From Theorem~\ref{THM:observations_table_symmetric_spaces} we know that the Hermitian spaces of $\rank \geq 2$ and $b(G/K)\leq (\dim G/K) /2$ are

\begin{center}
 $\grp{2}{n+1}{\R}$, $\gr{2}{n+1}{\C}$, $\gr{3}{6}{\C}$, $\so{10}/\u{5}$, $\so{12}/\u{6}$, $\so{14}/\u{7}$,  $\mathsf{E}_6/\spin{10} \u{1}$ and $\mathsf{E}_7/\mathsf{E}_6\u{1}$.
\end{center} 
 
In the case of $\grp{2}{n+1}{\R}$, the total space $\so{n+1}/\so{n-1}$ of the corresponding fat bundle equals the unit tangent bundle $T^1\sph^{n}$. As mentioned in Subsection~\ref{SS:fat_homogeneous_bundles}, the Euler class of $\su{n+1}/(\su{2}\times\su{n-1})\to \gr{2}{n+1}{\C}$ is a generator of $H^2(\gr{2}{n+1}{\C};\ZZ)= \ZZ$; this follows from \cite[Lemma 12.2]{BKS15} since both spaces are simply connected.

The third set is similar to the previous one. Recall that a compact irreducible symmetric space $G/K$ is quaternionic if $K$ locally splits off an $\sp{1}$ factor $K=K'\sp{1}$ and the isotropy representation of this $\sp{1}$ factor is equivalent to the standard Hopf action on $\H^n$ (see \cite[Chapter 14.E]{Besse} for their classification). The bundle $G/K'\to G/K'\sp{1}$ is fat and principal, with the group $L$ acting being $\sp{1}$ or $\so{3}$; see \cite[Section~8, Cas (B)]{Be75} or \cite[p.~17]{Zi99}. Morever, the associated $2$-sphere homogeneous bundle $G/K' \times_{L}\sph^2\to G/K'\sp{1}$ is also fat, see Cas~(A) and Cas~(B) of Sections 7 and 8 in \cite{Be75} or \cite[Proposition 2.22]{Zi99}. %It follows that there is an associated quadruple $K' <  K'\u{1} <  K'\sp{1} < G$ inducing fat bundles  
%$$
%G/K'\to G/K'\u{1}=G/K' \times_{L}\sph^2\to G/K'\sp{1}
%$$ 
%where the two arrows as well as their composition are fat bundles \cite[p.~55, Remarque]{Be75}. 
The quaternionic spaces of $\rank \geq 2$ and $b(G/K)\leq (\dim G/K) /2$ are 

\begin{center}
$\gr{2}{n+1}{\C}$, $\grp{2}{6}{\R}$, $\mathsf{F}_4/\sp{3}\sp{1}$, $\mathsf{E}_6/\su{6}\su{2}$, $\mathsf{E}_7/\spin{12}\sp{1}$, $\mathsf{E}_8/\mathsf{E}_7\sp{1}$ and $\mathsf{G}_2/\so{4}$.
\end{center} 

For each such $G/K$ there is exactly one associated $\sp{1}$ bundle (and the corresponding $2$-sphere bundle) as above, except in the case $\mathsf{G}_2/\so{4}$ \cite[Section~7, Cas~(A) and Cas~(B)]{Be75}. The group $\so{4}$ has two normal subgroups $\su{2}^{\pm}$ corresponding to the image of $\su{2}\times \{1\}$ and $\{1\} \times \su{2}$ under the covering $\su{2}\times \su{2}\to \so{4}$, and both $\su{2}^{\pm}$ can be enlarged to subgroups $\su{2}^{\pm} < \u{2}^{\pm} < \so{4}$; see \cite[pp.~1--2]{KZ17}. In the case of $\gr{2}{n+1}{\C}$, the total space equals $W_{1,1}^{4n-1}\cong T^1\CP^n$ and the associated $2$-sphere bundle equals the projectivized tangent bundle $\mathbb{P}_\CC T\CP^n$ of $\CP^n$. 

We remark that $\mathbb{P}_\CC T\CP^n$ is not only related to $W_{1,1}^{4n-1}$ but to any of the spaces $W_{p,q}^{4n-1}$. In fact, for any $p,q$ there is a triple
\begin{equation}\label{EQ:AW_circle_bundle}
\sss (\u{1}^{p,q}\times \u{n-1}) <\sss (\u{1}\times\u{1}\times \u{n-1}) < \su{n+1}
\end{equation}
inducing a circle bundle $W_{p,q}^{4n-1} \to \mathbb{P}_\CC T\CP^n$. For our purposes it is not relevant whether these circle bundles are fat, but they will be very useful for topological computations in Subsection~\ref{SS:fat_bundles_topology}. 

The following theorem is obtained by applying Theorem~\ref{THM:Wallach_generalization} to the fat bundles above, in combination with the values $b(-)$ from Table~\ref{table:b_symmetric} at the end of the article. In order to keep track of the spaces in the theorem, observe that each table corresponds to the respective set of spaces discussed above, with the exception of $\mathbb{P}_\CC T\CP^n$ being included in the first set (since, as explained in the previous paragraph, it is associated to any of the $W_{p,q}^{4n-1}$).

\begin{theorem}\label{THM:homogeneous_spaces_fat}
Each homogeneous space $G/H$ in the following tables carries a homogeneous metric of $\Ric_k>0$ for the corresponding value of $k$ (we restrict to $n\geq 2$ and distinguish between the cases $n\neq 3$ and $n = 3$ where it corresponds):%, $n\neq 3$
\renewcommand{\arraystretch}{1.3}
\begin{center}
$
\begin{array}{c c c c} 
\hline
\multirow{2}{*}{\vspace{1ex}$G/H$} & \multirow{2}{*}{\vspace{1ex}$\dim G/H$} &  \multicolumn{2}{c}{k}     \\[-1.3ex] %\cline{3-4}
&&{\scriptstyle (n\neq 3)} & {\scriptstyle (n= 3)}
\\  
\hline
W_{p,q}^{4n-1}  & 4n-1  &  2n-3 & 4\\ %[0.5ex] 
\hline 
\mathbb{P}_\CC T\CP^n   & 4n-2  &  2n-3 & 4\\ %[0.5ex] 
\hline
\mathbb{P}_\HH T\HP^n   & 8n-4  &  4n-7& 6\\ %[0.5ex] 
\hline
\su{6}/\sss (\u{2}\times \sp{2})   & 21  &  \multicolumn{2}{c}{7} \\ %[0.5ex] 
\hline
\so{10}/(\so{2}\times \spin{7})  & 23  &  \multicolumn{2}{c}{8}\\ %[0.5ex]
\hline
\end{array}
$
\end{center}
\vspace{-1ex}
\begin{center}
\begin{tabular}{c@{\qquad}c}
$
\begin{array}{c c c c } 
& & & \\
\hline
\multirow{2}{*}{\vspace{1ex}$G/H$} & \multirow{2}{*}{\vspace{1ex}$\dim G/H$} &  \multicolumn{2}{c}{k}     \\[-1.3ex] %\cline{3-4}
&&{\scriptstyle (n\neq 3)} & {\scriptstyle (n= 3)}
\\ 
\hline
T^1\sph^{n}   & 2n-1  &  n-1 & 3  \\ %[0.5ex] 
\hline
\su{n+1}/(\su{2}\times\su{n-1}) & 4n-3 & 2n-3 & 4\\ [0.5ex] 
\hline 
\su{6}/(\su{3}\times\su{3})   & 19  &  \multicolumn{2}{c}{9}  \\ %[0.5ex] 
  \hline
  \so{10}/\su{5}  & 21 & \multicolumn{2}{c}{7} \\ %[0.5ex] 
\hline
 \so{12}/\su{6}  & 31 & \multicolumn{2}{c}{15} \\ %[0.5ex] 
\hline
\so{14}/\su{7} & 43 & \multicolumn{2}{c}{21} \\ %[0.5ex] 
\hline 
  \mathsf{E}_6/\spin{10}   & 33  &  \multicolumn{2}{c}{11} \\ %[0.5ex] 
\hline
 \mathsf{E}_7/\mathsf{E}_6   & 55  &  \multicolumn{2}{c}{27}  \\ %[0.5ex] 
\hline
\end{array}
$
&
$
\begin{array}{c c c} 
\hline
 G/H & \dim G/H  &  k \\  
\hline
\so{6}/(\so{2}\times\su{2}) & 11 & 4 \\
\hline
\mathsf{F}_4/\sp{3} & 31  &  13  \\ %[0.5ex] 
\hline
\mathsf{E}_6/\su{6} & 43  & 19  \\ %[0.5ex] 
\hline
\mathsf{E}_7/\spin{12} & 67  &  31  \\ %[0.5ex] 
\hline
\mathsf{E}_8/\mathsf{E}_7 & 115  &  55  \\ %[0.5ex] 
\hline
 \mathsf{G}_2/\su{2}^{\pm}   & 11  &  3  \\ %[0.5ex]  
\hline 
\end{array}
$
\end{tabular}

%\hline
%  \mathsf{G}_2/\u{2}^{\pm}   & 10  &  3   \\ [0.5ex]

\end{center}
Moreover, for each $G/H$ in the table on the right the associated $2$-sphere bundle $G/H \times_{L}\sph^2$ with $L=\sp{1}$ or $\so{3}$ carries a homogeneous metric of $\Ric_k>0$ for the same $k$.
%\begin{samepage}
%\begin{itemize}
%\item The spaces $T^1\sph^3$, $W_{p,q}^{11}$, $\mathbb{P}_\CC T\CP^3$, $\su{4}/(\su{2}\times\su{2})$ and $\mathbb{P}_\HH T\HP^3$ admit homogeneous metrics of $\Ric_k>0$ for $k=3$, $4$, $4$, $4$ and $6$, respectively.
%\item 
%\end{itemize} 
%\end{samepage}
\end{theorem}

\begin{remark}\label{rem:T1S3}
	As stated in Theorem~\ref{THM:homogeneous_simple_version}, $T^1\sph^3=\so{4}/\so{2}$ actually admits an $\so{4}$-homo\-ge\-neous metric of $\Ric_2>0$. Such metric is induced from the one obtained in $\sp{1}\times\sp{1}$ by applying Theorem~\ref{THM:Wilkings_products} to $G=\sp{1}$. Indeed, let $\pi\colon \sp{1}\times\sp{1}\to \so{4}$ be the Lie group covering map that sends a pair $(p,q)$ of unit quaternions to the orthogonal map $\pi(p,q)$ of $\R^4$ given by $\pi(p,q)v\defeq pvq^{-1}$, for each $v\in \R^4\cong \H=\spann_{\mathbb{R}}\{1,i,j,k\}$. The image under $\pi$ of the diagonal circle subgroup $\Delta \sph^1=\{(e^{i\theta},e^{i\theta}):\theta\in\R\}<\sp{1}\times\sp{1}$ leaves $\spann_{\mathbb{R}}\{1,i\}\subset\H$ pointwise fixed, and hence $\pi(\Delta\sph^1)=\{\diag(1,1,A):A\in\so{2}\}<\so{4}$. Thus, $\pi$ induces a diffeomorphism $[\pi]\colon (\sp{1}\times\sp{1})/\Delta \sph^1 \to\so{4}/\so{2}$ which is $\sp{1}\times\sp{1}$-equivariant, since $\pi$ is a Lie group homomorphism. Therefore, the submersion metric on $(\sp{1}\times\sp{1})/\Delta \sph^1$ of $\Ric_2>0$ determines an $\so{4}$-invariant metric of $\Ric_2>0$ on $\so{4}/\so{2}=T^1\sph^3$.	
\end{remark}

We finish with some topological comments. The spaces $\mathsf{G}_2/\su{2}^+$ and $\mathsf{G}_2/\u{2}^+$ have the rational cohomology ring (but not the integral cohomology ring) of $\sph^{11}$ and $\CP^5$ respectively; see \cite{KZ17} and \cite[pp.~195-196]{Be78}. The integral cohomology rings of several homogeneous spaces of exceptional Lie groups have been computed in \cite[Section~5.1]{Du13}.

%%%%%%%%%%%%%%%%%%%%%%%%%%%%%%%%%%%%%

\subsection{Topology}\label{SS:fat_bundles_topology}

Here we compute various topological properties of the spaces $W_{p,q}^{4n-1}$. Throughout this section all cohomology groups are understood to be taken with integer coefficients. As in Subsection~\ref{SS:fat_homogeneous_bundles}, $p,q$ are assumed to be coprime integers with $pq>0$ (although all computations also work for the case $pq<0$). The main tool will be the Gysin sequence associated to the following circle bundles induced from the triples \eqref{EQ:AW_circle_bundle}:
$$
\pi_{p,q}\colon W_{p,q}^{4n-1} \to \mathbb{P}_\CC T\CP^n.
$$
The cohomology ring of $\mathbb{P}_\CC T\CP^n$ was computed by Borel in \cite[Proposition 31.1]{Bo53}:
$$
H^*(\mathbb{P}_\CC T\CP^n)=\frac{S(x_1)\otimes S(x_2)\otimes S(x_3,\dots ,x_{n+1})}{S^+(x_1,\dots ,x_{n+1})},
$$
where $S(-)$ denotes symmetric polynomials in the corresponding variables (of degree $2$) and $S^+(-)\subset S(-)$ denotes those of positive degree. In particular, $H^*(\mathbb{P}_\CC T\CP^n)$ has no torsion and is concentrated in even degrees, with non-trivial groups given by:
$$
H^{2k}(\mathbb{P}_\CC T\CP^n)=\begin{cases}
\ZZ^{k+1},  &\text{ if } k< n,\\
\ZZ^{2n-k}, &\text{ if } n\leq k\leq 2n-1. \\
\end{cases}
$$
The cohomology ring can be rewritten as follows:
\begin{equation}\label{EQ:cohomology_PTCPn}
H^*(\mathbb{P}_\CC T\CP^n)=\frac{\ZZ[x,y]}{
\begin{cases}
x^{n}+x^{n-1}y+\dots +xy^{n-1}+y^n=0\\
xy(x^{n-1}+x^{n-2}y+\dots +xy^{n-2}+y^{n-1})=0 \\
\end{cases}
},
\end{equation}
where $x,y$ are of degree $2$. Moreover, as explained in \cite[pp.~473-474]{KS91} for the case $n=2$, the generators $x,y$ can be chosen so that the Euler class $e(\pi_{p,q})$ of $\pi_{p,q}$ equals
$$
e(\pi_{p,q})=-qx+py\in H^2(\mathbb{P}_\CC T\CP^n).
$$
With this information at hand one can fully compute the cohomology ring of $W_{p,q}^{4n-1}$ using the Gysin sequence, which gives exact sequences 
$$
0\to H^{2k-1}(W_{p,q}^{4n-1})\to H^{2k-2}(\mathbb{P}_\CC T\CP^n)\xrightarrow{e\cup} H^{2k}(\mathbb{P}_\CC T\CP^n)\to H^{2k}(W_{p,q}^{4n-1})\to 0.
$$
For our purposes it will sufficient to look at cohomology groups of degree $\leq 2n$. Multiplication by the Euler class $e\cup$ is injective for $2k\leq 2n$, thus it follows that the cohomology groups $H^{2k-1}(W_{p,q}^{4n-1})$ vanish. Consequently we get isomorphisms
\begin{equation}\label{EQ:quotient_euler_cohomology}
H^{2k}(W_{p,q}^{4n-1})=\frac{H^{2k}(\mathbb{P}_\CC T\CP^n)}{e\cup H^{2k-2}(\mathbb{P}_\CC T\CP^n)}.
\end{equation}
For $2k<2n$ there are no relations among the elements $x^k,x^{k-1}y,\dots,xy^{k-1},y^k$ from \eqref{EQ:cohomology_PTCPn}. Since $p,q$ are coprime with $pq\neq 0$ it follows that the quotient \eqref{EQ:quotient_euler_cohomology} is isomorphic to $\ZZ$.

The case $2k=2n$ is bit more involved, so let us have a closer look at the map $e\cup$. Using the first relation in \eqref{EQ:cohomology_PTCPn} we can give the following bases to the spaces involved:
\begin{align*}
H^{2n-2}(\mathbb{P}_\CC T\CP^n) &= \langle  x^{n-1},x^{n-2}y,\dots,xy^{n-2},y^{n-1}  \rangle \cong\ZZ^n, \\
H^{2n}(\mathbb{P}_\CC T\CP^n) &= \langle  x^{n-1}y,x^{n-2}y^2,\dots,xy^{n-1},y^{n}  \rangle \cong\ZZ^n.
\end{align*}
The $n\times n$ matrix representing the map $e\cup$ with respect to these bases equals:
$$
T_n(p,q)\defeq 
{\tiny
\begin{pmatrix}
p+q & -q & 0 & \cdots & 0 & 0\\
q & p & -q & \cdots & 0 & 0\\
q & 0 & p & \cdots & 0 & 0\\
\vdots & \vdots & \vdots & \ddots & \vdots & \vdots\\
q & 0 & 0 & \cdots & p & -q\\
q & 0 & 0 & \cdots & 0 & p\\
\end{pmatrix}
}.
$$
The quotient in \eqref{EQ:quotient_euler_cohomology} for $k=n$ is a finite group whose order equals the absolute value of the determinant of the matrix $T_n(p,q)$. We denote the latter by $\tau_n(p,q)$:
$$
\tau_n(p,q)\defeq \vert \det T_n(p,q) \vert = \vert p^n+p^{n-1}q+\dots + pq^{n-1}+q^n\vert .
$$
% =\vert \frac{p^{n+1}-q^{n+1}}{p-q}\vert
Observe that $\tau_n(p,q)>0$ if $pq>0$. Let us summarize the discussion above.

\begin{lemma}\label{LEM:topology_Aloff_Wallach}
The manifolds $W_{p,q}^{4n-1}$ are simply connected and the first $2n$ cohomology groups equal:
$$
H^{2k}(W_{p,q}^{4n-1})=\begin{cases}
\ZZ, & \text{ if } k \text{ even and} <2n,\\
\Gamma_{\tau_n(p,q)}, & \text{ if } k=2n,\\
0, & \text{ if } k \text{ odd and } <2n,\\
\end{cases}
$$
where $\Gamma_{\tau_n(p,q)}$ is some finite group of order $\tau_n(p,q)$.
\end{lemma}

The fundamental group of $W_{p,q}^{4n-1}$ can be computed using \cite[Lemma 12.2]{BKS15} in terms of the Euler class: since $p,q$ are coprime, the element $e(\pi_{p,q})=-qx+py$ is indivisible in $H^2(\mathbb{P}_\CC T\CP^n)$ and hence the total space $W_{p,q}^{4n-1}$ is simply connected. The order of $H^{2n}(W_{p,q}^{4n-1})$ can be extracted from Wilking's article \cite[p.~119]{Wi02}. He also observed that $\tau_n(p,q)= 0$ may occur if one allows $pq<0$. For the case $n=2$ the reader can find further details in the article of Kreck and Stolz \cite[pp. 473-474]{KS91}.

%More precisely, $\tau_n(p,q)= 0$ if and only if $4n-1$ is of the form $8m+3$ and $p=1=-q$, and in this case $H^{2n}(W_{p,q}^{4n-1})=\ZZ$. 

%% file: tables.tex
%\setcounter{section}{6}

\newgeometry{top=2cm, bottom=2cm, left=2.5cm, right=2.5cm}
\LTcapwidth=\textheight
\begin{landscape}
	\thispagestyle{empty}
	\scriptsize
	\pdfbookmark{Table: Positive intermediate Ricci curvature for irreducible symmetric spaces}{}
\begin{longtable}{cccccccc}
	%\begin{table}\captionsetup{width=70em}
	\caption{For each irreducible symmetric space $M$, we list the symmetric spaces $B_{\Phi^k}$, where $\Phi^k=\Lambda\setminus\{\alpha_k\}$, along with $b(M)=1+\max\{\dim B_{\Phi^k}:1\leq k\leq r\}$ and the index $k_{\text{max}}$ of the a simple root $\alpha_{k_{\text{max}}}$ for which the maximum is achieved. The symmetric spaces $M$ and $B_{\Phi^k}$ are listed up to finite quotients and coverings.}\label{table:b_symmetric}\\
	\renewcommand{\arraystretch}{1.3}
	%\begin{tabular}{lllllllll}
	%	\hline
		Dynkin diagram & Type & Multiplicities & $M$ & $\dim M$ & $B_{\Phi^k}$, $\dim B_{\Phi^k}$ & $k_{\text{max}}$ & $b(M)$ 
		\\ \hline
		\endfirsthead
		\caption{(Continuation)}
		\\
		Dynkin diagram & Type & Multiplicities & $M$ & $\dim M$ & $B_{\Phi^k}$, $\dim B_{\Phi^k}$ & $k_{\text{max}}$ & $b(M)$ 
		\\ \hline
		\endhead
		\multirow{8}{*}{
		\begin{tikzpicture}[node distance=\nodedistance,g/.style={circle,inner sep=\circleinnersep,draw},a0/.style={rectangle,inner sep=\rectangleinnersep,draw}]
		\node[g] (a1) [label=below:$\alpha_1$] {};
		\node[g] (a2) [right=of a1,label=below:$\alpha_2$] {}
		edge [] (a1);
		\node[] (a3) [right=of a2] {};
		\node[g] (ap) [right=of a3, label=below:$\alpha_{r-1}$] {};
		\node[g] (an) [right=of ap, label=below:$\alpha_r$] {}
		edge [] (ap);
		\draw ($(a3)!.5!(ap)$) -- (ap);
		\draw ($(a3)!.5!(a2)$) -- (a2);
		\draw [dashed] ($(a3)!.5!(a2)$) -- ($(a3)!.5!(ap)$);
		\end{tikzpicture}	
	} 
		& $\sf{A}_1$ & $n$ & $\mathbb{S}^{n+1}$  & $n+1$ & $\{\text{point}\}$, $0$ & $1$ & $1$ 
		\\ \cline{2-8}
		& \multirow{2}{*}{$\sf{A}_r$}  & \multirow{2}{*}{$1,\dots,1$} & \multirow{2}{*}{$\su{r+1}/\so{r+1}$} & \multirow{2}{*}{$\frac{1}{2}r(r+3)$} & $\su{k}/\so{k}\times \su{r-k+1}/\so{r-k+1}$, & \multirow{2}{*}{$1$} &  \multirow{2}{*}{$\frac{1}{2}r(r+1)$}
		\\
		&&&&& $\frac{1}{2}((k-1)(k+2)+(r-k)(r-k+3))$
		\\ \cline{2-8}
		& \multirow{2}{*}{$\sf{A}_r$}& \multirow{2}{*}{$2,\dots,2$} & \multirow{2}{*}{$\su{r+1}$} & \multirow{2}{*}{$r(r+2)$} & $\su{k}\times \su{r-k+1}$, & \multirow{2}{*}{$1$} &  \multirow{2}{*}{$r^2$}
		\\
		&&&&& $k^2-1+(r-k)(r-k+2)$
		\\ \cline{2-8}
		&\multirow{2}{*}{$\sf{A}_r$} & \multirow{2}{*}{$4,\dots,4$} & \multirow{2}{*}{$\su{2r+2}/\sp{r+1}$} & \multirow{2}{*}{$r(2r+3)$}  & $\su{2k}/\sp{k}\times \su{2r-2k+2}/\sp{r-k+1}$, & \multirow{2}{*}{$1$}&  \multirow{2}{*}{$2r^2-r$}
		\\
		&&&&& $(k-1)(2k+1)+(r-k)(2r-2k+3)$
		\\ \cline{2-8}
		&$\sf{A}_2$ & $8,8$ & $\sf{E}_6/\sf{F}_4$ & $26$ & $\mathbb{S}^9$, $9$ & $1$ & $10$
		\\ \hline
		\multirow{4}{*}{
		\begin{tikzpicture}[node distance=\nodedistance,g/.style={circle,inner sep=\circleinnersep,draw},a0/.style={rectangle,inner sep=\rectangleinnersep,draw}]
		\node[g] (a2) [label=below:$\alpha_1$] {};
		\node[] (a3) [right=of a2] {};
		\node[g] (an2) [right=of a3, label=below:$\alpha_{r-2}$] {};
		\node[g] (an1) [right=of an2, label=below:$\alpha_{r-1}$] {}
		edge [] (an2);
		\node[g] (an) [right=of an1, label=below:$\alpha_r$] {};
%		\draw ($(an1)!.65!(an)$) -- ($(an1)!.65!(an)+(-0.3,0.2)$);
%		\draw ($(an1)!.65!(an)$) -- ($(an1)!.65!(an)+(-0.3,-0.2)$);
%		\draw [] (an1.north east) to (an.north west);
%		\draw [] (an1.south east) to (an.south west);
		\draw [double, double distance=\circleinnersep, arrows={-Implies}] (an1.east) to (an.west);
		\draw ($(a3)!.5!(a2)$) -- (a2);
		\draw ($(a3)!.5!(an2)$) -- (an2);
		\draw [dashed] ($(a3)!.5!(a2)$) -- ($(a3)!.5!(an2)$);
		\end{tikzpicture}	
	}
		&\multirow{2}{*}{$\sf{B}_r$} & \multirow{2}{*}{$1,\dots, 1, n$} & \multirow{2}{*}{$\grp{r}{2r+n}{\R}$, $n\geq 1$} & \multirow{2}{*}{$r(r+n)$} & $\su{k}/\so{k}\times\grp{r-k}{2r-2k+n}{\R}$, & \multirow{2}{*}{$1$} & \multirow{2}{*}{$r^2+(n-2)(r-1)$}
		\\
		&&&&& $\frac{1}{2}(k-1)(k+2)+(r-k)(r-k+n)$
		\\ \cline{2-8}
		&\multirow{2}{*}{$\sf{B}_r$} & \multirow{2}{*}{$2,\dots, 2, 2$} & \multirow{2}{*}{$\so{2r+1}$} & \multirow{2}{*}{$r(2r+1)$} & $\su{k}\times \so{2r-2k+1}$, & \multirow{2}{*}{$1$} & \multirow{2}{*}{$2r^2-3r+2$}
		\\
		&&&&& $k^2-1+(r-k)(2r-2k+1)$
		\\
		\hline
			\multirow{13}{*}{
		\begin{tikzpicture}[node distance=\nodedistance,g/.style={circle,inner sep=\circleinnersep,draw},a0/.style={rectangle,inner sep=\rectangleinnersep,draw}]
		\node[g] (a1) [label=below:$\alpha_1$] {};
		\node[] (a3) [right=of a1] {};
		\node[g] (an2) [right=of a3, label=below:$\alpha_{r-2}$] {};
		\node[g] (an1) [right=of an2, label=below:$\alpha_{r-1}$] {}
		edge [] (an2);
		\node[g] (an) [right=of an1, label=below:$\alpha_r$] {};
%		\draw ($(an1)!.35!(an)$) -- ($(an1)!.35!(an)+(0.3,0.2)$);
%		\draw ($(an1)!.35!(an)$) -- ($(an1)!.35!(an)+(0.3,-0.2)$);
%		\draw [] (an1.north east) to (an.north west);
%		\draw [] (an1.south east) to (an.south west);
		\draw [double, double distance=\circleinnersep, arrows={-Implies}] (an.west) to (an1.east);
		\draw ($(a3)!.5!(a1)$) -- (a1);
		\draw ($(a3)!.5!(an2)$) -- (an2);
		\draw [dashed] ($(a3)!.5!(a1)$) -- ($(a3)!.5!(an2)$);
		\end{tikzpicture}	
	}
		& \multirow{2}{*}{$\sf{C}_r$} & \multirow{2}{*}{$1,\dots, 1, 1$} & \multirow{2}{*}{$\sp{r}/\u{r}$} & \multirow{2}{*}{$r(r+1)$} & $\su{k}/\so{k}\times \sp{r-k}/\u{r-k}$, & \multirow{2}{*}{$1$} & \multirow{2}{*}{$r^2-r+1$}
		\\
		&&&&& $\frac{1}{2}(k-1)(k+2)+(r-k)(r-k+1)$
		\\ \cline{2-8}
		& \multirow{2}{*}{$\sf{C}_r$}
		& \multirow{2}{*}{$2,\dots, 2,1$}
		& \multirow{2}{*}{$\gr{r}{2r}{\C}$}
		& \multirow{2}{*}{$2r^2$} 
		& $\su{k}\times\gr{r-k}{2r-2k}{\C}$,
		& $1$ & $2r^2-4r+3$ ($r\geq 3$)
		\\ %\cline{7-8} 
		&&&&& $k^2-1+2(r-k)^2$ & $r$ & $r^2$ ($r\leq 2$)
		\\ \cline{2-8}
		& \multirow{2}{*}{$\sf{C}_r$} & \multirow{2}{*}{$2,\dots, 2,2$} & \multirow{2}{*}{$\sp{r}$} & \multirow{2}{*}{$r(2r+1)$} & $\su{k}\times\sp{r-k}$, & \multirow{2}{*}{$1$} & \multirow{2}{*}{$2r^2-3r+2$}
		\\
		&&&&& $k^2-1+(r-k)(2r-2k+1)$
		\\ \cline{2-8}
		& \multirow{2}{*}{$\sf{C}_r$} 
		& \multirow{2}{*}{$4,\dots, 4,1$} 
		& \multirow{2}{*}{$\so{4r}/\u{2r}$} 
		& \multirow{2}{*}{$2r(2r-1)$} 
		& $\su{2k}/\sp{k}\times\so{4r-4k}/\u{2r-2k}$,
		& $1$ & $4r^2-10r+7$ ($r\geq 4$)
		\\ %\cline{7-8} 
		&&&&& $(k-1)(2k+1)+2(r-k)(2r-2k-1)$ & $r$ & $2r^2-r$ ($r\leq 3$)
		\\ \cline{2-8}
		& \multirow{2}{*}{$\sf{C}_r$}
		& \multirow{2}{*}{$4,\dots, 4,3$}
		& \multirow{2}{*}{$\gr{r}{2r}{\H}$}
		& \multirow{2}{*}{$4r^2$}
		& $\su{2k}/\sp{k}\times\gr{r-k}{2r-2k}{\H}$,
		& $1$ & $4r^2-8r+5$ ($r\geq 3$)
		\\ %\cline{7-8} 
		&&&&& $(k-1)(2k+1)+4(r-k)^2$ & $r$ & $2r^2-r$ ($r\leq 2$)
		\\ \cline{2-8}
		& \multirow{3}{*}{$\sf{C}_3$}
		& \multirow{3}{*}{$8,8,1$}
		& \multirow{3}{*}{$\sf{E}_7/\sf{E}_6\u1$}
		& \multirow{3}{*}{$54$}
		& $\grp{2}{12}{\R}$, $20$ ($k=1$) 
		& \multirow{3}{*}{$3$}
		& \multirow{3}{*}{$27$}
		\\
		&&&&& $\mathbb{S}^9\times\mathbb{S}^2$, $11$ ($k=2$)&
		\\
		&&&&& $\sf{E}_6/\sf{F}_4$, $26$ ($k=3$)
		\\
		\hline
		\pagebreak
%		\end{tabular}
%	\end{table}
%
%	\pagebreak
%
%	\begin{table}
%		\begin{tabular}{lllllllll}
%		Dynkin diagram & Type & Multiplicities & $M$ & $\dim M$ & $B_{\Phi^k}$, $\dim B_{\Phi^k}$ & $k_{\text{max}}$ & $b(M)$ 
%		\\ \hline
		\multirow{4}{*}{
		\begin{tikzpicture}[node distance=\nodedistance, g/.style={circle,inner sep=\circleinnersep,draw},a0/.style={rectangle,inner sep=\rectangleinnersep,draw}]
		\node[g] (a2) [label=below:$\alpha_1$] {};
		\node[] (a3) [right=of a2] {};
		\node[g] (an3) [right=of a3, label=below:$\alpha_{r-3}$] {};
		\node[g] (an2) [right=of an3, label=below:$\alpha_{r-2}$] {}
		edge [] (an3);
		\node[g] (an1) at ($(an2)+(-25:0.7)$) [label=right:$\alpha_{r-1}$] {}
		edge [] (an2);
		\node[g] (an) at ($(an2)+(25:0.7)$) [label=right:$\alpha_r$] {}
		edge [] (an2);
		\draw ($(a3)!.5!(a2)$) -- (a2);
		\draw ($(a3)!.5!(an3)$) -- (an3);
		\draw [dashed] ($(a3)!.5!(a2)$) -- ($(a3)!.5!(an3)$);
		\end{tikzpicture}
		}
		& \multirow{2}{*}{$\sf{D}_r$}
		& \multirow{2}{*}{$1,\dots, 1,1,1$}
		& \multirow{2}{*}{$\grp{r}{2r}{\R}$}
		& \multirow{2}{*}{$r^2$}
		& $\su{k}/\so{k}\times\grp{r-k}{2r-2k}{\R}$,
		& $1$ & $r^2-2r+2$ ($r\geq 4$)
		\\ 
		&&&&& $\frac{1}{2}(k-1)(k+2)+(r-k)^2$ & $r$ & $\frac{1}{2}(r^2+r)$ ($r=3$)
		\\ \cline{2-8}
		& \multirow{2}{*}{$\sf{D}_r$}
		& \multirow{2}{*}{$2,\dots, 2,2,2$}
		& \multirow{2}{*}{$\so{2r}$}
		& \multirow{2}{*}{$r(2r-1)$}
		& $\su{k}\times\so{2r-2k}$,
		& $1$ & $2r^2-5r+4$ ($r\geq 4$)
		\\
		&&&&& $k^2-1+(r-k)(2r-2k-1)$ & $r$ & $r^2$ ($r = 3$)
				\\ \hline
		\multirow{9}{*}{
			\begin{tikzpicture}[node distance=\nodedistance,g/.style={circle,inner sep=\circleinnersep,draw},a0/.style={circle,inner sep=\circleinnersep,draw,double,double distance=0.6*\circleinnersep}]
			\node[g] (a2) [label=below:$\alpha_1$] {};
			\node[] (a3) [right=of a2] {};
			\node[g] (an2) [right=of a3, label=below:$\alpha_{r-2}$] {};
			\node[g] (an1) [right=of an2, label=below:$\alpha_{r-1}$] {}
			edge [] (an2);
%			\node[a0] (anbis) [right=of an1] {};
			\node[a0] (an) [right=of an1, label=below:$\alpha_r$] {};
%			\draw ($(an1)!.65!(an)$) -- ($(an1)!.65!(an)+(-0.3,0.2)$);
%			\draw ($(an1)!.65!(an)$) -- ($(an1)!.65!(an)+(-0.3,-0.2)$);
%			\draw [] (an1.north east) to (an.north west);
%			\draw [] (an1.south east) to (an.south west);
			\draw ($(a3)!.5!(a2)$) -- (a2);
			\draw ($(a3)!.5!(an2)$) -- (an2);
			\draw [dashed] ($(a3)!.5!(a2)$) -- ($(a3)!.5!(an2)$);
			\draw [double, double distance=\circleinnersep, arrows={Implies-Implies}] (an.west) to (an1.east);
			\end{tikzpicture}	
		}
		& \multirow{2}{*}{$\sf{BC}_r$} & \multirow{2}{*}{$2,\dots, 2, 2n[1]$} & \multirow{2}{*}{$\gr{r}{2r+n}{\C}$, $n\geq 1$} & \multirow{2}{*}{$2r(r+n)$} & $\su{k}\times\gr{r-k}{2r-2k+n}{\C}$, & \multirow{2}{*}{$1$} & \multirow{2}{*}{$2r^2+2(n-2)r+3-2n$}
		\\
		&&&&& $k^2-1+2(r-k)(r-k+n)$
		\\ \cline{2-8}
		& \multirow{2}{*}{$\sf{BC}_r$} & \multirow{2}{*}{$4,\dots, 4, 4n[3]$} & \multirow{2}{*}{$\gr{r}{2r+n}{\H}$, $n\geq 1$} & \multirow{2}{*}{$4r(r+n)$} & $\su{2k}/\sp{k}\times\gr{r-k}{2r-2k+n}{\H}$, & \multirow{2}{*}{$1$} & \multirow{2}{*}{$4r^2+4(n-2)r+5-4n$}
		\\
		&&&&& $(k-1)(2k+1)+4(r-k)(r-k+n)$
		\\ \cline{2-8}
		& \multirow{2}{*}{$\sf{BC}_r$} & \multirow{2}{*}{$4,\dots, 4, 4[1]$} & \multirow{2}{*}{$\so{4r+2}/\u{2r+1}$ }& \multirow{2}{*}{$2r(2r+1)$} & $\su{2k}/\sp{k}\times \so{4r-4k+2}/\u{2r-2k+1}$, & \multirow{2}{*}{$1$} & \multirow{2}{*}{$4r^2-6r+3$}
		\\
		&&&&& $(k-1)(2k+1)+(r-k)(4r-4k+2)$
		\\ \cline{2-8}
		& \multirow{2}{*}{$\sf{BC}_2$}
		& \multirow{2}{*}{$6, 8[1]$}
		& \multirow{2}{*}{$\sf{E}_6/\sf{Spin}_{10}\sf{U}_1$}
		& \multirow{2}{*}{$32$}
		& $\CP^5$, $10$ ($k=1$) 
		& \multirow{2}{*}{$1$}
		& \multirow{2}{*}{$11$}
		\\
		&&&&& $\mathbb{S}^7$, $7$ ($k=2$)
		\\ \cline{2-8}
		& $\sf{BC}_1$ & $8[7]$ & $\sf{F}_4/\spin{9}$ & $16$ & $\{\text{point}\}$, $0$ & $1$ & $1$
		\\ 
		\hline
%	\end{tabular}
%	\end{table}
\end{longtable}
\pagebreak
\addtocounter{table}{-1}
\begin{longtable}{ccccccccccc}
	%\begin{table}\captionsetup{width=70em}
	\caption{(Continuation)}\\%\label{table:b_symmetric}\\
	\renewcommand{\arraystretch}{1.3}
	%\begin{tabular}{lllllllll}
	%	\hline
		Dynkin diagram & Type & Multiplicities & $M$ & $\dim M$ & $k$ & $B_{\Phi^k}$ & $\dim B_{\Phi^k}$ & $k_{\text{max}}$ & $b(M)$ 
	\\ \hline
	\endfirsthead
	\caption{(Continuation)}
	\\
		Dynkin diagram & Type & Multiplicities & $M$ & $\dim M$ & $k$ & $B_{\Phi^k}$ & $\dim B_{\Phi^k}$ & $k_{\text{max}}$ & $b(M)$ 
	\\ \hline
	\endhead	
%	\begin{table}
%	\begin{tabular}{lllllllllll}
%		Dynkin diagram & Type & Multiplicities & $M$ & $\dim M$ & $k$ & $B_{\Phi^k}$ & $\dim B_{\Phi^k}$ & $k_{\text{max}}$ & $b(M)$ 
%		\\
%		\hline
		\multirow{2}{*}{
			\begin{tikzpicture}[node distance=\nodedistance, g/.style={circle,inner sep=\circleinnersep,draw},a0/.style={rectangle,inner sep=\rectangleinnersep,draw}]
			\node[g] (a1) [label=below:$\alpha_1$] {};
			\node[g] (a2) [right=of a1, label=below:$\alpha_2$] {}
			edge []  (a1);
			%		\draw ($(a1)!.35!(a2)$) -- ($(a1)!.35!(a2)+(0.3,0.2)$);
			%		\draw ($(a1)!.35!(a2)$) -- ($(a1)!.35!(a2)+(0.3,-0.2)$);
			%		\draw [] (a1.north east) to (a2.north west);
			%		\draw [] (a1.south east) to (a2.south west);
			\draw [double, double distance=1.7*\circleinnersep, arrows={-Implies}] (a2.west) to (a1.east);
			\draw [] (a1.east) to (a2.west);
			\end{tikzpicture}
		}
		& $\sf{G}_2$ & $1,1$ & $\sf{G}_2/\so{4}$ & $8$ & $1,2$ & $\mathbb{S}^2$ & $2$ & $1$ & $3$
		\\ \cline{2-10}
		& $\sf{G}_2$ & $2,2$ & $\sf{G}_2$ & $14$ & $1,2$ & $\su{2}$ & $3$ & $1$ & $4$
		\\ 
		\hline
		\multirow{18}{*}{
			\begin{tikzpicture}[node distance=\nodedistance, g/.style={circle,inner sep=\circleinnersep,draw},a0/.style={rectangle,inner sep=\rectangleinnersep,draw}]
			\node[g] (a1) [label=below:$\alpha_1$] {};
			\node[g] (a2) [right=of a1, label=below:$\alpha_2$] {}
			edge [] (a1);
			\node[g] (a3) [right=of a2, label=below:$\alpha_3$] {};
			\node[g] (a4) [right=of a3, label=below:$\alpha_4$] {}
			edge [] (a3);
			\draw [double, double distance=\circleinnersep, arrows={-Implies}] (a2.east) to (a3.west);
			%		\draw ($(a2)!.65!(a3)$) -- ($(a2)!.65!(a3)+(-0.3,0.2)$);
			%		\draw ($(a2)!.65!(a3)$) -- ($(a2)!.65!(a3)+(-0.3,-0.2)$);
			%		\draw [] (a2.north east) to (a3.north west);
			%		\draw [] (a2.south east) to (a3.south west);
			\end{tikzpicture}
		}
		& \multirow{3}{*}{$\sf{F}_4$}
		& \multirow{3}{*}{$1,1,1,1$}
		& \multirow{3}{*}{$\sf{F}_4/\sp{3}\sp{1}$}
		& \multirow{3}{*}{$28$}
		& $1$
		& $\sp{3}/\u{3}$
		& $12$
		& \multirow{3}{*}{$1$}
		& \multirow{3}{*}{$13$}
		\\
		&&&&& $2,3$
		& $\mathbb{S}^2\times\su{3}/\so{3}$ 
		& $7$
		\\
		&&&&& $4$ 
		& $\grp{3}{7}{\R}$ 
		& $12$
		\\ \cline{2-10}
		& \multirow{4}{*}{$\sf{F}_4$}
		& \multirow{4}{*}{$1,1,2,2$}
		& \multirow{4}{*}{$\sf{E}_6/\su{6}\su{2}$}
		& \multirow{4}{*}{$40$}
		& $1$
		& $\gr{3}{6}{\C}$
		& $18$
		& \multirow{4}{*}{$1$} 
		& \multirow{4}{*}{$19$}
		\\
		&&&&& $2$
		& $\mathbb{S}^2\times\su{3}$
		& $10$
		\\
		&&&&& $3$
		& $\su{2}\times \su{3}/\so{3}$
		& $8$
		\\
		&&&&& $4$
		&$\grp{3}{8}{\R}$
		& $15$
		\\ \cline{2-10}
		& \multirow{4}{*}{$\sf{F}_4$}
		& \multirow{4}{*}{$1,1,4,4$}
		& \multirow{4}{*}{$\sf{E}_7/\spin{12}\sp{1}$}
		& \multirow{4}{*}{$64$}
		& $1$
		& $\so{12}/\u{6}$
		& $30$
		& \multirow{4}{*}{$1$} 
		& \multirow{4}{*}{$31$}
		\\
		&&&&& $2$
		& $\mathbb{S}^2\times\su{6}/\sp{3}$
		& $16$
		\\
		&&&&& $3$
		& $\su{3}/\so{3}\times\mathbb{S}^5$
		& $10$
		\\
		&&&&& $4$
		&$\grp{3}{10}{\R}$
		& $21$
		\\ \cline{2-10}
		& \multirow{4}{*}{$\sf{F}_4$}
		& \multirow{4}{*}{$1,1,8,8$}
		& \multirow{4}{*}{$\sf{E}_8/\sf{E}_{7}\sp{1}$}
		& \multirow{4}{*}{$112$}
		& $1$
		& $\mathsf{E}_7/\mathsf{E}_6\u{1}$
		& $54$
		& \multirow{4}{*}{$1$} 
		& \multirow{4}{*}{$55$}
		\\
		&&&&& $2$
		& $\mathbb{S}^2\times\mathsf{E}_6/\mathsf{F}_4$
		& $28$
		\\
		&&&&& $3$
		& $\su{3}/\so{3}\times \mathbb{S}^9$
		& $14$
		\\
		&&&&& $4$
		&$\grp{3}{14}{\R}$
		& $33$
		\\ \cline{2-10}
		& \multirow{3}{*}{$\sf{F}_4$}
		& \multirow{3}{*}{$2,2,2,2$}
		& \multirow{3}{*}{$\sf{F}_4$}
		& \multirow{3}{*}{$52$}
		& $1$
		& $\sp{3}$
		& $21$
		& \multirow{3}{*}{$1$}
		& \multirow{3}{*}{$22$}
		\\
		&&&&& $2,3$
		& $\su{2}\times\su{3}$ 
		& $11$
		\\
		&&&&& $4$ 
		& $\so{7}$ 
		& $21$
		\\
		\hline
		\multirow{8}{*}{
		\begin{tikzpicture}[node distance=\nodedistance, g/.style={circle,inner sep=\circleinnersep,draw},a0/.style={rectangle,inner sep=\rectangleinnersep,draw}]
		\node[g] (a1) [label=above:$\alpha_6$] {};
		\node[g] (a3) [left=of a1, label=above:$\alpha_5$] {}
		edge [] (a1);
		\node[g] (a4) [left=of a3, label=85:$\alpha_4$] {}
		edge [] (a3);
		\node[g] (a2) [above=of a4, label=left:$\alpha_2$] {}
		edge [] (a4);
		\node[g] (a5) [left=of a4, label=above:$\alpha_3$] {}
		edge [] (a4);
		\node[g] (a6) [left=of a5, label=above:$\alpha_1$] {}
		edge [] (a5);
		\end{tikzpicture}
		}
		& \multirow{4}{*}{$\sf{E}_6$}
		& \multirow{4}{*}{$1,\dots, 1$}
		& \multirow{4}{*}{$\sf{E}_6/(\sp4/\mathbb{Z}_2)$}
		& \multirow{4}{*}{$42$}
		& $1,6$
		& $\grp{5}{10}{\R}$
		& $25$
		& \multirow{4}{*}{$1$}
		& \multirow{4}{*}{$26$}
		\\
		&&&&& $3,5$ 
		& $\mathbb{S}^2\times \su{5}/\so{5}$ 
		& $16$
		\\
		&&&&& $4$
		& $(\su{3}/\so{3})^2\times \mathbb{S}^2$
		& $12$
		\\
		&&&&& $2$ & $\su{6}/\so{6}$
		& $20$
		\\ \cline{2-10}
		& \multirow{4}{*}{$\sf{E}_6$}
		& \multirow{4}{*}{$2,\dots, 2$}
		& \multirow{4}{*}{$\sf{E}_6$}
		& \multirow{4}{*}{$78$} 
		& $1,6$
		& $\so{10}$
		& $45$
		& \multirow{4}{*}{$1$}
		& \multirow{4}{*}{$46$}
		\\
		&&&&& $3,5$
		& $\su{2}\times\su{5}$
		& $27$
		\\
		&&&&& $4$
		& $\su{3}^2\times \su{2}$
		& $19$
		\\
		&&&&& $2$ & $\su{6}$ & $35$
		\\ \hline
%		\end{tabular}
%		\end{table}
	
		\pagebreak

%		\begin{table}
%		\begin{tabular}{lllllllllll}
%		Dynkin diagram & Type & Multiplicities & $M$ & $\dim M$ & $k$ & $B_{\Phi^k}$ & $\dim B_{\Phi^k}$ & $k_{\text{max}}$ & $b(M)$ 
%				\\
%		\hline
		\multirow{14}{*}{
			\begin{tikzpicture}[node distance=\nodedistance, g/.style={circle,inner sep=\circleinnersep,draw},a0/.style={rectangle,inner sep=\rectangleinnersep,draw}]
			\node[g] (a1) [label=above:$\alpha_7$] {};
			\node[g] (a3) [left=of a1, label=above:$\alpha_6$] {}
			edge [] (a1);
			\node[g] (a4) [left=of a3, label=85:$\alpha_5$] {}
			edge [] (a3);
			\node[g] (a2) [above=of a4, label=left:$\alpha_2$] {}
			edge [] (a4);
			\node[g] (a5) [left=of a4, label=above:$\alpha_4$] {}
			edge [] (a4);
			\node[g] (a6) [left=of a5, label=above:$\alpha_3$] {}
			edge [] (a5);
			\node[g] (a7) [left=of a6, label=above:$\alpha_1$] {}
			edge [] (a6);
			\end{tikzpicture}
		}
		& \multirow{7}{*}{$\sf{E}_7$}
		& \multirow{7}{*}{$1,\dots, 1$}
		& \multirow{7}{*}{$\sf{E}_7/(\su8/\mathbb{Z}_2)$}
		& \multirow{7}{*}{$70$}
		& $1$
		& $\sf{E}_6/(\sp{4}/\mathbb{Z}_2)$ 
		& $42$
		& \multirow{7}{*}{$1$}
		& \multirow{7}{*}{$43$}
		\\
		&&&&& $2$
		& $\su{7}/\so{7}$
		& $27$
		\\
		&&&&& $3$
		& $\mathbb{S}^2\times \grp{5}{10}{\R}$ & $27$
		\\
		&&&&& $4$ 
		& $\su{3}/\so{3} \times \su{5}/\so{5}$ & $19$
		\\
		&&&&& $5$ 
		& $\mathbb{S}^2\times \su{3}/\so{3}\times \su{4}/\so{4}$ & $16$
		\\
		&&&&& $6$
		& $\mathbb{S}^2\times \su{6}/\so{6}$ 
		& $22$
		\\
		&&&&& $7$ 
		& $\grp{6}{12}{\R}$
		& $36$
		\\ \cline{2-10}
		& \multirow{7}{*}{$\sf{E}_7$}
		& \multirow{7}{*}{$2,\dots, 2$}
		& \multirow{7}{*}{$\sf{E}_7$}
		& \multirow{7}{*}{$133$}
		& $1$
		& $\sf{E}_6$ 
		& $78$
		& \multirow{7}{*}{$1$}
		& \multirow{7}{*}{$79$}
		\\
		&&&&& $2$
		& $\su{7}$
		& $48$
		\\
		&&&&& $3$
		& $\su{2}\times \so{10}$ & $48$
		\\
		&&&&& $4$ 
		& $\su{3} \times \su{5}$ & $32$
		\\
		&&&&& $5$ 
		& $\su{2}\times \su{3}\times \su{4}$ & $26$
		\\
		&&&&& $6$
		& $\su{2}\times \su{6}$ 
		& $38$
		\\
		&&&&& $7$ 
		& $\so{12}$
		& $66$
		\\ \hline
		\multirow{16}{*}{
		\begin{tikzpicture}[node distance=\nodedistance, g/.style={circle,inner sep=\circleinnersep,draw},a0/.style={rectangle,inner sep=\rectangleinnersep,draw}]
		\node[g] (a1) [label=above:$\alpha_8$] {};
		\node[g] (a3) [left=of a1, label=above:$\alpha_7$] {}
		edge [] (a1);
		\node[g] (a4) [left=of a3, label=85:$\alpha_6$] {}
		edge [] (a3);
		\node[g] (a2) [above=of a4, label=left:$\alpha_2$] {}
		edge [] (a4);
		\node[g] (a5) [left=of a4, label=above:$\alpha_5$] {}
		edge [] (a4);
		\node[g] (a6) [left=of a5, label={above:$\alpha_4$}] {}
		edge [] (a5);
		\node[g] (a7) [left=of a6, label=above:$\alpha_3$] {}
		edge [] (a6);
		\node[g] (a8) [left=of a7, label=above:$\alpha_1$] {}
		edge [] (a7);
		\end{tikzpicture}
		}
		& \multirow{8}{*}{$\sf{E}_8$}
		& \multirow{8}{*}{$1,\dots, 1$}
		& \multirow{8}{*}{$\sf{E}_8/(\spin{16}/\mathbb{Z}_2)$}
		& \multirow{8}{*}{$128$}
		& $1$
		& $\mathsf{E}_7/(\su{8}/\mathbb{Z}_2)$ 
		& $70$
		& \multirow{8}{*}{$1$}
		& \multirow{8}{*}{$71$}
		\\
		&&&&& $2$
		& $\su{8}/\so{8}$
		& $35$
		\\
		&&&&& $3$
		& $\mathbb{S}^2\times \mathsf{E}_6/(\sp{4}/\mathbb{Z}_2)$ & $44$
		\\
		&&&&& $4$ 
		& $\su{3}/\so{3} \times \grp{5}{10}{\R}$ & $30$
		\\
		&&&&& $5$ 
		& $\su{4}/\so{4}\times \su{5}/\so{5}$ & $23$
		\\
		&&&&& $6$
		& $\mathbb{S}^2\times \su{5}/\so{5}\times\su{3}/\so{3}$ 
		& $21$
		\\
		&&&&& $7$ 
		& $\su{7}/\so{7}\times\mathbb{S}^2$
		& $29$
		\\
		&&&&& $8$ 
		& $\grp{7}{14}{\R}$
		& $49$
		\\ \cline{2-10}
		& \multirow{8}{*}{$\sf{E}_8$}
		& \multirow{8}{*}{$2,\dots, 2$}
		& \multirow{8}{*}{$\sf{E}_8$}
		& \multirow{8}{*}{$248$}
		& $1$
		& $\sf{E}_7$ 
		& $133$
		& \multirow{8}{*}{$1$}
		& \multirow{8}{*}{$134$}
		\\
		&&&&& $2$
		& $\su{8}$
		& $63$
		\\
		&&&&& $3$
		& $\su{2}\times \mathsf{E}_6$ & $81$
		\\
		&&&&& $4$ 
		& $\su{3} \times \so{10}$ & $53$
		\\
		&&&&& $5$ 
		& $\su{4}\times \su{5}$ & $39$
		\\
		&&&&& $6$
		& $\su{2}\times \su{5}\times \su{3}$ 
		& $35$
		\\
		&&&&& $7$ 
		& $\su{7}\times\su{2}$
		& $51$
		\\
		&&&&& $8$ 
		& $\so{14}$
		& $91$
		\\
		\hline
%	\end{tabular}
%	\end{table}
\end{longtable}
\end{landscape}
\restoregeometry